\documentclass[12pt]{amsart}
\usepackage[latin1]{inputenc}
\usepackage{amssymb,amsmath,amsthm,comment,enumerate,hyperref,multicol,mathdots,tikz,ytableau}
\usetikzlibrary{cd}

\theoremstyle{definition}
\newtheorem{theorem}{Theorem}[section]
\newtheorem{lemma}[theorem]{Lemma}
\newtheorem{proposition}[theorem]{Proposition}

\newtheorem{corollary}[theorem]{Corollary}

\newtheorem{definition}[theorem]{Definition}

\newtheorem{remark}[theorem]{Remark}

\definecolor{UF2}{HTML}{FA4616}
\definecolor{UF}{HTML}{0021A5}
\newcommand{\cross}{\begin{tikzpicture}[scale = .4]
	\draw[lightgray] (0,0) -- (1,0) -- (1,1) -- (0,1) -- (0,0);
	\draw[thick,UF!100!black] (.5,0) -- (.5,1);
	\draw[thick,UF!100!black] (0,.5) -- (1,.5);
\end{tikzpicture}}
\newcommand{\hwire}{\begin{tikzpicture}[scale = .4]
	\draw[lightgray] (0,0) -- (1,0) -- (1,1) -- (0,1) -- (0,0);
	\draw[thick,UF!100!black] (0,.5) -- (1,.5);
\end{tikzpicture}}
\newcommand{\vwire}{\begin{tikzpicture}[scale = .4]
	\draw[lightgray] (0,0) -- (1,0) -- (1,1) -- (0,1) -- (0,0);
	\draw[thick,UF!100!black] (.5,0) -- (.5,1);
\end{tikzpicture}}
\newcommand{\nowire}{\begin{tikzpicture}[scale = .4]
	\draw[lightgray] (0,0) -- (1,0) -- (1,1) -- (0,1) -- (0,0);
\end{tikzpicture}}
\newcommand{\shade}{\begin{tikzpicture}[scale = .4]
	\draw[fill=lightgray] (0,0) -- (1,0) -- (1,1) -- (0,1) -- (0,0);
\end{tikzpicture}}

\newcommand{\bump}{\begin{tikzpicture}[scale = .4]
	\draw[lightgray] (0,0) -- (1,0) -- (1,1) -- (0,1) -- (0,0);
	\draw[thick,UF!100!black] (0.5,0) to[out=90,in=180] (1,0.5);
	\draw[thick,UF!100!black] (0,0.5) to[out=0,in=-90] (0.5,1);
\end{tikzpicture}}
\newcommand{\jay}{\begin{tikzpicture}[scale = .4]
	\draw[lightgray] (0,0) -- (1,0) -- (1,1) -- (0,1) -- (0,0);
	\draw[thick,UF!100!black] (0,0.5) to[out=0,in=-90] (0.5,1);
\end{tikzpicture}}

\newcommand{\are}{\begin{tikzpicture}[scale = .4]
	\draw[lightgray] (0,0) -- (1,0) -- (1,1) -- (0,1) -- (0,0);
	\draw[thick,UF!100!black] (0.5,0) to[out=90,in=180] (1,0.5);
\end{tikzpicture}}
\newcommand{\inlinecross}{\begin{tikzpicture}[scale = .3]
	\draw[lightgray] (0,0) -- (1,0) -- (1,1) -- (0,1) -- (0,0);
	\draw[thick,UF!100!black] (.5,0) -- (.5,1);
	\draw[thick,UF!100!black] (0,.5) -- (1,.5);
\end{tikzpicture}}
\newcommand{\inlinenowire}{\begin{tikzpicture}[scale = .3]
	\draw[lightgray] (0,0) -- (1,0) -- (1,1) -- (0,1) -- (0,0);
\end{tikzpicture}}
\newcommand{\inlinebump}{\begin{tikzpicture}[scale = .3]
	\draw[lightgray] (0,0) -- (1,0) -- (1,1) -- (0,1) -- (0,0);
	\draw[thick,UF!100!black] (0.5,0) to[out=90,in=180] (1,0.5);
	\draw[thick,UF!100!black] (0,0.5) to[out=0,in=-90] (0.5,1);
\end{tikzpicture}}
\newcommand{\inlineare}{\begin{tikzpicture}[scale = .3]
	\draw[lightgray] (0,0) -- (1,0) -- (1,1) -- (0,1) -- (0,0);
	\draw[thick,UF!100!black] (0.5,0) to[out=90,in=180] (1,0.5);
\end{tikzpicture}}
\newcommand{\inlinejay}{\begin{tikzpicture}[scale = .3]
	\draw[lightgray] (0,0) -- (1,0) -- (1,1) -- (0,1) -- (0,0);
	\draw[thick,UF!100!black] (0,0.5) to[out=0,in=-90] (0.5,1);
\end{tikzpicture}}
\newcommand{\inlineshade}{\begin{tikzpicture}[scale = .3]
	\draw[fill=lightgray] (0,0) -- (1,0) -- (1,1) -- (0,1) -- (0,0);
\end{tikzpicture}}
\newcommand{\inlinehwire}{\begin{tikzpicture}[scale = .3]
	\draw[lightgray] (0,0) -- (1,0) -- (1,1) -- (0,1) -- (0,0);
	\draw[thick,UF!100!black] (0,.5) -- (1,.5);
\end{tikzpicture}}
\newcommand{\inlinevwire}{\begin{tikzpicture}[scale = .3]
	\draw[lightgray] (0,0) -- (1,0) -- (1,1) -- (0,1) -- (0,0);
	\draw[thick,UF!100!black] (.5,0) -- (.5,1);
\end{tikzpicture}}

\begin{document}

\title[]{Lenart's bijection via bumpless pipe dreams}
\author[A. Gregory]{Adam Gregory}
\address[AG]{Department of Mathematics, University of Florida, Gainesville, FL 32601}
\email{adamgregory@ufl.edu}

\author[Z. Hamaker]{Zachary Hamaker}
\address[ZH]{Department of Mathematics, University of Florida, Gainesville, FL 32601}
\email{zhamaker@ufl.edu}

\date{\today}

\begin{abstract}
Pipe dreams and bumpless pipe dreams for vexillary permutations are each known to be in bijection with certain semistandard tableaux via maps due to Lenart and Weigandt, respectively.
Recently, Gao and Huang have defined a bijection between the former two sets.
In this note we show for vexillary permutations that the Gao--Huang bijection preserves the associated tableaux, giving a new proof of Lenart's result.
Our methods extend to give a recording tableau for any bumpless pipe dream.
\end{abstract}

\maketitle

\section{Introduction}

Schubert polynomials are fundamental objects in algebraic combinatorics and enumerative algebraic geometry.
They have several combinatorial formulas, most notably the Billey-Jockusch-Stanley pipe dream formula~\cite{BB93,billey1993some}.
Schubert polynomials associated to vexillary permutations are flagged Schur functions, which have a formula in terms of flagged tableaux~\cite{wachs1985flagged}.
As part of an effort to extend this type of formula to all Schubert polynomials, Lenart described a bijection from pipe dreams to flagged tableaux~\cite[Rem.~4.12~(2)]{L04}.
Lenart's map relies on Edelman-Greene insertion~\cite{edelman1987balanced}, and he does not give a proof, remarking `the proof is somewhat technical, and therefore we decided to omit it.'
A special case of Lenart's map appears as~\cite[Thm.~3.3]{SS12}.

Recently Lam, Lee and Shimozono introduced a new combinatorial model for Schubert polynomials called the bumpless pipe dream formula~\cite{lam2021back}.
Building on work of Kreiman~\cite{kreiman2006schubert} and Knutson-Miller-Yong~\cite{knutson2009grobner}, Weigandt observed that bumpless pipe dreams for vexillary permutations are manifestly in bijection with flagged tableaux~\cite[Thm. 7.7]{weigandt2021bumpless}.
Later, Gao and Huang gave a bijection between bumpless pipe dreams and pipe dreams~\cite{gao2021canonical}.
By composing these maps, one obtains a new bijection from pipe dreams to flagged tableaux.
See Figure~\ref{fig:bijections} for an example of these maps.

Our main result is:
\begin{theorem}[See Theorem~\ref{t:main}]
\label{t:intro}
The composition of the Weigandt and Gao-Huang bijections is equal to Lenart's bijection.
\end{theorem}

As a consequence, we obtain the first explicit proof of Lenart's bijection, avoiding many of the technicalities required in a direct approach at the expense of drawing on a larger body of prior work.
To prove Theorem~\ref{t:intro}, we first demonstrate the Grassmannian case (see Theorem~\ref{t:Grassmannian}).
We then use combinatorial recurrences on pipe dreams~\cite{billey2019bijective,little2003combinatorial} and bumpless pipe dreams~\cite{huang2020bijective} to extend this to the vexillary case.
In doing so, we make critical use of the `canonical' nature of the Gao-Huang bijection and a result of the second author and Young~\cite{hamaker2014relating}.
From this perspective, our work extends the Little bijection~\cite{little2003combinatorial} to give an intrinsic description of `recording' tableaux for bumpless pipe dreams (see Section~\ref{s:recording}).
A direct generalization of this map will appear in forthcoming work by Huang, Shimozono and Yu~\cite{private}.

\

\noindent \textbf{Outline:}
In the second section, we introduce the objects studied in this paper and their properties: permutations and reduced words, flagged tableaux, reduced compatible sequences~\cite{billey1993some} and pipe dreams~\cite{BB93}, bumpless pipe dreams~\cite{lam2021back,weigandt2021bumpless}.
The third section describes properties of the various maps that we require: the Edelman-Greene correspondence~\cite{edelman1987balanced} and Lenart's bijection~\cite{L04}, the Gao-Huang bijection between pipe dreams and bumpless pipe dreams~\cite{gao2021canonical},  the Little bijection~\cite{little2003combinatorial}, Huang's combinatorial realization of Monk's formula~\cite{huang2020bijective} (precise definitions appear in the appendix).
We prove our main results in the fourth section.
The fifth section discusses an extension of~\cite{huang2020bijective} and some forthcoming work.

\ 

\noindent \textbf{Acknowledgements:}
We thank Joshua Arroyo, Daoji Huang and Tianyi Yu for helpful conversations.
We are especially grateful to Hugh Dennin for sharing his code for bumpless pipe dreams and related maps and to Anna Weigandt for suggesting this problem.

\section{Objects}

Let $[n] = \{1,2,\dots,n\}$.

\subsection{Permutations}


For our purposes, a \emph{permutation} $w$ is a bijection of the positive integers so that $w(k) = k$ for all but finitely many $k$.
Let $S_\infty$ denote the set of all permutations.
The \emph{support} of $w \in S_\infty$ is $\text{supp}(w) = \{k: w(k) \neq k \}$.
With these conventions, the usual symmetric group is $S_n = \{w \in S_\infty: \text{supp}(w) \subseteq [n]\}$.
For $w \in S_n$, we have the one-line notation $w = w(1) \ldots w(n)$. 


Given integers $j > i \geq 1$, let $t_{ij}$ be the \emph{transposition} of the integers $i$ and $j$.
Similarly, let $s_i = t_{i\:i+1}$ denote a \emph{simple transposition}.
It is well-known that every $w \in S_\infty$ can be written as a product of simple transpositions. 
For $w = s_{a_1} \dots s_{a_p}$ an expression of a minimal length, we call $\textbf{a} = (a_1,\dots,a_p)$ a \emph{reduced word} for $w$.
Let $\text{Red}(w)$ denote the set of reduced words for $w$ and $\ell(w)$ be the \emph{length} of a reduced word.

For $w \in S_\infty$, a \emph{descent} is a value $k$ such that $w(k) > w(k+1)$.
A permutation $w$ is called \emph{Grassmannian} if it has at most one descent.
A permutation $w$ is \emph{vexillary} if for all $1 \leq i < j < k < \ell$, we do not have $w(j) < w(i) < w(\ell) < w(k)$.
Equivalently, vexillary permutations are those avoiding the pattern 2143.
Note every Grassmannian permutation is also vexillary.

The \emph{code} of $w \in S_\infty$ is $c(w) = (c_1(w), \ldots, c_n(w))$ where 
\[c_i(w) = |\{j : j > i \; \text{and} \; w(i) > w(j)\}|.
\]
The \emph{shape} of $w$ is the integer partition $\lambda(w)$ obtained by rearranging $c(w)$ to be weakly decreasing. 
For example, $c(1432) = (0,2,1,0)$ and so $\lambda(1432) = (2,1)$.

\subsection{Diagrams and Tableaux}

A \emph{diagram} is a subset of $[n] \times [n]$. We use matrix coordinates for these arrays, so row indices increase from top-to-bottom and column indices increase from left-to-right. 
The \emph{Rothe diagram} of a permutation $w \in S_n$ is
\[
D_w = \{(i,w(j)) : i < j \; \text{and} \; w(i) > w(j) \}.
\]
The \emph{Young diagram} of an integer partition $\lambda = (\lambda_1, \lambda_2, \ldots)$ given by
\[
Y_\lambda = \{(i,j) : 1 \leq j \leq \lambda_i\}.
\]

A \emph{Young tableau} is a positive integer filling of a Young diagram.
It is \emph{semi-standard} if its filling increases weakly across rows and strictly down columns.
The \emph{shape} of a tableau is its underlying partition.
For $\lambda = (\lambda_1,\dots,\lambda_\ell)$ a partition, let $\text{SSYT}(\lambda)$ denote the set of semi-standard Young tableaux with shape $\lambda$. 
Similarly, $\text{SSYT}_k(\lambda)$ is the subset of semi-standard tableaux with maximum entry ${\leq k}$.
More generally, for $\phi = (\phi_1, \ldots, \phi_\ell)$ positive integers,  define
\[
\text{SSYT}_\phi(\lambda) = \{T \in \text{SSYT}(\lambda): T_{i\lambda_i} \leq \phi_i \ \mbox{for}\ i \in [\ell]\}
\]
where $T_{ij}$ is the filling in $(i,j)$.
We call $\phi$ a \emph{flag} and $\text{SSYT}_\phi(\lambda)$ the set of \emph{$\phi$-flagged semi-standard Young tableau} of shape $\lambda$.
For example,

\[\text{SSYT}_{(2,3)}(2,1) = \left\{\ytableausetup{smalltableaux}
\begin{ytableau}
1 & 1 \\ 2
\end{ytableau}, \quad 
\begin{ytableau}
1 & 1 \\ 3
\end{ytableau}, \quad 
\begin{ytableau}
1 & 2 \\ 2
\end{ytableau}, \quad 
\begin{ytableau}
1 & 2 \\ 3
\end{ytableau}, \quad 
\begin{ytableau}
2 & 2 \\ 3
\end{ytableau} 
\right\}.
\]

Each vexillary permutation has a flag $\phi(v) = (\phi_1, \ldots, \phi_m)$ associated to it in the following way.
Let $v$ be a vexillary permutation with Rothe diagram $D_v$ and shape $\lambda(v) = (\lambda_1, \ldots, \lambda_m)$.
Then $\phi_i$ is equal to the row index of the southeastern-most box in $D_v$ which lies in the same diagonal as $(i,\lambda_i)$.
If $v$ is Grassmannian with $\text{Des}(v) = \{k\}$ one can check that $\phi(v) = (k, \ldots, k)$.
For example, with $w = 35142$ we see $\text{code}(w) = (2,3,0,1,0)$ so that $\lambda(w) = (3,2,1)$.
Then $\phi(w) = (2,2,4)$:

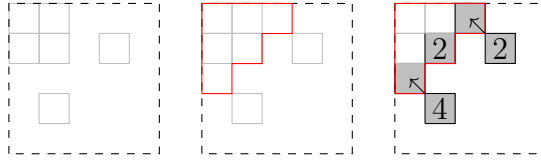
\[
\begin{tikzpicture}[scale=0.4]
\node at (1,1) {\nowire};
\node at (0,3) {\nowire};
\node at (0,4) {\nowire};
\node at (1,3) {\nowire};
\node at (1,4) {\nowire};
\node at (3,3) {\nowire};
\draw[dashed] (-0.5,-0.5) -- (-0.5,4.5) -- (4.5,4.5) -- (4.5,-0.5) -- (-0.5,-0.5);
\end{tikzpicture} \quad 
\begin{tikzpicture}[scale=0.4]
\node at (1,1) {\nowire};
\node at (0,3) {\nowire};
\node at (0,4) {\nowire};
\node at (1,3) {\nowire};
\node at (1,4) {\nowire};
\node at (3,3) {\nowire};
\draw[dashed] (-0.5,-0.5) -- (-0.5,4.5) -- (4.5,4.5) -- (4.5,-0.5) -- (-0.5,-0.5);
\draw[red] (-0.5,1.5) -- (0.5,1.5) -- (0.5,2.5) -- (1.5,2.5) -- (1.5,3.5) -- (2.5,3.5) -- (2.5,4.5) -- (-0.5,4.5) -- (-0.5,1.5);
\end{tikzpicture} \quad 
\begin{tikzpicture}[scale=0.4]
\node at (1,1) {\shade};
\node at (0,2) {\shade};
\node at (0,3) {\nowire};
\node at (0,4) {\nowire};
\node at (1,3) {\shade};
\node at (1,4) {\nowire};
\node at (2,4) {\shade};
\node at (3,3) {\shade};
\node at (1,1) {4};
\node at (1,3) {2};
\node at (3,3) {2};

\draw[dashed] (-0.5,-0.5) -- (-0.5,4.5) -- (4.5,4.5) -- (4.5,-0.5) -- (-0.5,-0.5);
\draw[red] (-0.5,1.5) -- (0.5,1.5) -- (0.5,2.5) -- (1.5,2.5) -- (1.5,3.5) -- (2.5,3.5) -- (2.5,4.5) -- (-0.5,4.5) -- (-0.5,1.5);
\draw[->] (2.5,3.5) -- (2,4);
\draw[->] (0.5,1.5) -- (0,2);
\end{tikzpicture}
\]

For partitions $\mu$ and $\lambda$ say $\mu \subseteq \lambda$ if $Y_\mu \subseteq Y_\lambda$.
Given $\mu \subseteq \lambda$, the associated \emph{skew diagram} is $Y_{\lambda /\mu} = Y_\lambda \setminus Y_\mu$.
A \emph{skew tableau} of shape $\lambda / \mu$ is a filling of $Y_{\lambda /\mu}$ with positive integers.
For a tableau $T$ of shape $\lambda$ and $\mu \subseteq \lambda$, let $T/\mu$ be the skew tableau obtained by restricting $T$ to the skew diagram $Y_{\lambda / \mu}$. 

Let $T$ be a skew tableau of shape $\lambda / \mu$.
An \emph{outer corner} of $T$ is a rightmost box in some row of $Y_\lambda$.
An \emph{inner corner} of $T$ is a rightmost box in some row of $Y_\mu$.
\emph{Jeu de taquin} is an invertible process on skew tableaux that adds a distinguished inner corner and removes an outer corner according to the following sliding rule: starting with the inner corner, repeatedly compare the fillings to the right and below the empty tile and swap places with the smaller of the two, breaking ties by moving down and treating boxes outside the tableau as being filled with $\infty$.
To perform \emph{rectification} on $T$, do jeu de taquin one at a time to every cell in $Y_\mu$.
Let $\text{rect}(T)$ denote the rectification of $T$. 
See~\cite[\S1.2]{fulton1997young} for precise definitions of these maps.

Given a tableau $T$, let $\text{jdt}(T) = \text{rect}(T/(1))$.
For example, with

\[
T = 
\begin{ytableau}
1 & 1 & 3 & 4 \\
2 & 4 & 4 & 5\\
3 & 5
\end{ytableau}\]
we compute (with $\bullet$ indicating the empty tile)
\[
\begin{ytableau}
\none[\bullet] & 1 & 3 & 4 \\
2 & 4 & 4 & 5\\
3 & 5
\end{ytableau} \rightarrow
\begin{ytableau}
1 & \none[\bullet] & 3 & 4 \\
2 & 4 & 4 & 5 \\
3 & 5
\end{ytableau}  \rightarrow 
\begin{ytableau}
1 & 3 & \none[\bullet] & 4 \\
2 & 4 & 4 & 5 \\
3 & 5
\end{ytableau}  \rightarrow 
\begin{ytableau}
1 & 3 & 4 & 4 \\
2 & 4 & \none[\bullet] & 5 \\
3 & 5
\end{ytableau}  \rightarrow 
\begin{ytableau}
1 & 3 & 4 & 4 \\
2 & 4 & 5 & \none[\bullet] \\
3 & 5
\end{ytableau}
= \text{jdt}(T).
\]
Since jeu de taquin slides are invertible, we can invert $\text{jdt}(T)$ when we know the shape of $T$ and the filling removed from its top left corner.

\subsection{Reduced Compatible Sequences and Pipe Dreams} 

The Billey-Jockusch-Stanley formula for Schubert polynomials is given in terms of certain biwords called reduced compatible sequences~\cite{billey1993some}.
\begin{definition}
A \emph{reduced compatible sequence} for a permutation $w$ is a biword $\binom{\mathbf{r}}{\mathbf{a}} = \binom{r_1, \ldots, r_\ell}{a_1, \ldots, a_\ell}$ such that
\begin{itemize}
\item[(i)] $\mathbf{a}$ is a reduced word for $w$,
\item[(ii)] $r_i \leq r_{i+1}$ for each $1 \leq i \leq \ell$,
\item[(iii)] $r_i \leq a_i$ for each $1 \leq i \leq \ell$, and
\item[(iv)] $r_i < r_{i+1}$ whenever $a_i < a_{i+1}$.
\end{itemize}
\end{definition}

In~\cite{BB93}, Bergeron and Billey gave a diagrammatic representation of reduced compatible sequences that are now commonly known as pipe dreams.
For $n \geq 2$ let $\delta_n = (n-1,n-2,\ldots,2,1)$.
\begin{definition}
A (\emph{reduced}) \emph{pipe dream} is a filling of the staircase diagram $Y_{\delta_n}$ using the tiles 
\[
\cross \quad \text{and} \quad \bump
\]
in such a way that $n$ pipes enter from the top and exit to the left without any pair of pipes crossing more than once.
\end{definition}

\[
\begin{tikzpicture}[scale=0.4]
\node at (0,0) {\jay};

\node at (0,1) {\bump};
\node at (1,1) {\jay};

\node at (0,2) {\bump};
\node at (1,2) {\cross};
\node at (2,2) {\jay};

\node at (0,3) {\bump};
\node at (1,3) {\bump};
\node at (2,3) {\bump};
\node at (3,3) {\jay};

\node at (0,4) {\cross};
\node at (1,4) {\bump};
\node at (2,4) {\cross};
\node at (3,4) {\bump};
\node at (4,4) {\jay};

\draw[thick] (-0.5,-0.5) -- (-0.5,4.5) -- (4.5,4.5);

\node at (0,5) {1};
\node at (1,5) {2};
\node at (2,5) {3};
\node at (3,5) {4};
\node at (4,5) {5};

\node at (-1,4) {2};
\node at (-1,3) {1};
\node at (-1,2) {4};
\node at (-1,1) {5};
\node at (-1,0) {3};

\node at (2,-2) {Example};
\end{tikzpicture} \; \hspace{0.5in} \;
\begin{tikzpicture}[scale=0.4]
\node at (0,0) {\jay};

\node at (0,1) {\bump};
\node at (1,1) {\jay};

\node at (0,2) {\bump};
\node at (1,2) {\cross};
\draw[red,thick] (1,2) circle (0.4);
\node at (2,2) {\jay};

\node at (0,3) {\bump};
\node at (1,3) {\bump};
\node at (2,3) {\cross};
\draw[red,thick] (2,3) circle (0.4);
\node at (3,3) {\jay};

\node at (0,4) {\bump};
\node at (1,4) {\bump};
\node at (2,4) {\cross};
\node at (3,4) {\bump};
\node at (4,4) {\jay};

\draw[thick] (-0.5,-0.5) -- (-0.5,4.5) -- (4.5,4.5);

\node at (0,5) {1};
\node at (1,5) {2};
\node at (2,5) {3};
\node at (3,5) {4};
\node at (4,5) {5};

\node at (-1,4) {1};
\node at (-1,3) {2};
\node at (-1,2) {4};
\node at (-1,1) {5};
\node at (-1,0) {3};

\node at (2,-2) {Non-example};
\end{tikzpicture}
\]

Pipes are labeled $1$ through $n$ from left-to-right across the top, and the order in which they exit from top-to-bottom gives the \emph{permutation} of a pipe dream.
For $w \in S_\infty$, let $\text{PD}(w)$ denote the set of all pipe dreams whose permutation is $w$.

The correspondence between pipe dreams and reduced compatible sequences is as follows.
Given $P$ in $\text{PD}(w)$, map the locations $(i,j)$ of each $\inlinecross$-tile in $P$ to $\binom{i}{j+i-1}$ and concatenate these in the order of right-to-left and top-to-bottom.
For the inverse, the reduced compatible sequence $\binom{r_1, \ldots, r_\ell}{a_1, \ldots, a_\ell}$ maps to the pipe dream having $\inlinecross$-tiles in locations $(r_i, a_i - r_i + 1)$.

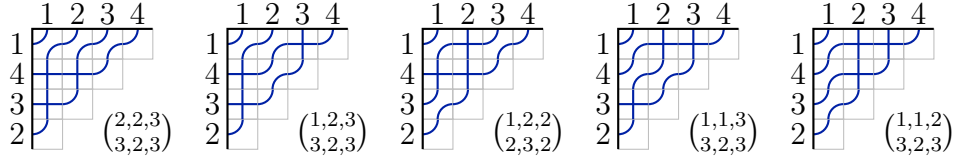
\[
\begin{tikzpicture}[scale=0.4]
\node at (0,0) {\jay};

\node at (0,1) {\cross};
\node at (1,1) {\jay};

\node at (0,2) {\cross};
\node at (1,2) {\cross};
\node at (2,2) {\jay};

\node at (0,3) {\bump};
\node at (1,3) {\bump};
\node at (2,3) {\bump};
\node at (3,3) {\jay};

\draw[thick] (-0.5,-0.5) -- (-0.5,3.5) -- (3.5,3.5);

\node at (0,4) {1};
\node at (1,4) {2};
\node at (2,4) {3};
\node at (3,4) {4};

\node at (-1,3) {1};
\node at (-1,2) {4};
\node at (-1,1) {3};
\node at (-1,0) {2};

\node at (3,0) {$\binom{2,2,3}{3,2,3}$};
\end{tikzpicture} \; 
\begin{tikzpicture}[scale=0.4]
\node at (0,0) {\jay};

\node at (0,1) {\cross};
\node at (1,1) {\jay};

\node at (0,2) {\cross};
\node at (1,2) {\bump};
\node at (2,2) {\jay};

\node at (0,3) {\bump};
\node at (1,3) {\bump};
\node at (2,3) {\cross};
\node at (3,3) {\jay};

\draw[thick] (-0.5,-0.5) -- (-0.5,3.5) -- (3.5,3.5);

\node at (0,4) {1};
\node at (1,4) {2};
\node at (2,4) {3};
\node at (3,4) {4};

\node at (-1,3) {1};
\node at (-1,2) {4};
\node at (-1,1) {3};
\node at (-1,0) {2};

\node at (3,0) {$\binom{1,2,3}{3,2,3}$};
\end{tikzpicture} \; 
\begin{tikzpicture}[scale=0.4]
\node at (0,0) {\jay};

\node at (0,1) {\bump};
\node at (1,1) {\jay};

\node at (0,2) {\cross};
\node at (1,2) {\cross};
\node at (2,2) {\jay};

\node at (0,3) {\bump};
\node at (1,3) {\cross};
\node at (2,3) {\bump};
\node at (3,3) {\jay};

\draw[thick] (-0.5,-0.5) -- (-0.5,3.5) -- (3.5,3.5);

\node at (0,4) {1};
\node at (1,4) {2};
\node at (2,4) {3};
\node at (3,4) {4};

\node at (-1,3) {1};
\node at (-1,2) {4};
\node at (-1,1) {3};
\node at (-1,0) {2};

\node at (3,0) {$\binom{1,2,2}{2,3,2}$};
\end{tikzpicture} \;
\begin{tikzpicture}[scale=0.4]
\node at (0,0) {\jay};

\node at (0,1) {\cross};
\node at (1,1) {\jay};

\node at (0,2) {\bump};
\node at (1,2) {\bump};
\node at (2,2) {\jay};

\node at (0,3) {\bump};
\node at (1,3) {\cross};
\node at (2,3) {\cross};
\node at (3,3) {\jay};

\draw[thick] (-0.5,-0.5) -- (-0.5,3.5) -- (3.5,3.5);

\node at (0,4) {1};
\node at (1,4) {2};
\node at (2,4) {3};
\node at (3,4) {4};

\node at (-1,3) {1};
\node at (-1,2) {4};
\node at (-1,1) {3};
\node at (-1,0) {2};

\node at (3,0) {$\binom{1,1,3}{3,2,3}$};
\end{tikzpicture} \; 
\begin{tikzpicture}[scale=0.4]
\node at (0,0) {\jay};

\node at (0,1) {\bump};
\node at (1,1) {\jay};

\node at (0,2) {\bump};
\node at (1,2) {\cross};
\node at (2,2) {\jay};

\node at (0,3) {\bump};
\node at (1,3) {\cross};
\node at (2,3) {\cross};
\node at (3,3) {\jay};

\draw[thick] (-0.5,-0.5) -- (-0.5,3.5) -- (3.5,3.5);

\node at (0,4) {1};
\node at (1,4) {2};
\node at (2,4) {3};
\node at (3,4) {4};

\node at (-1,3) {1};
\node at (-1,2) {4};
\node at (-1,1) {3};
\node at (-1,0) {2};

\node at (3,0) {$\binom{1,1,2}{3,2,3}$};
\end{tikzpicture}
\]
Due to this correspondence, we use pipe dreams and reduced compatible sequences interchangebly. 

\subsection{Bumpless Pipe Dreams}
Recently, Lam, Lee and Shimozono gave a new diagrammatic formula for Schubert polynomials in terms of diagrams called bumpless pipe dreams~\cite{lam2021back}.
As Weigandt observed~\cite{weigandt2021bumpless}, their formula is equivalent to a prior formula due to Lascoux~\cite{lascoux2008chern}.

\begin{definition}
A (\emph{reduced}) \emph{bumpless pipe dream} is a filling of an $n \times n$ grid using tiles from
\[
 \nowire \quad \vwire \quad \hwire \quad \are \quad \jay \quad \cross 
\]
in such a way that $n$ pipes enter from the bottom and exit to the right without any pair of pipes crossing more than once.
\end{definition}

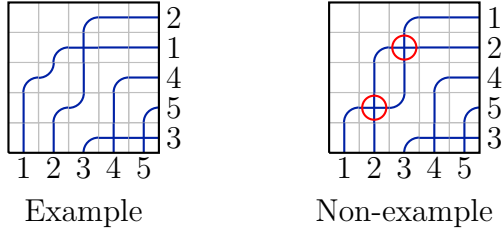
\[
\begin{tikzpicture}[scale=0.4]
\node at (0,0) {\vwire};
\node at (1,0) {\vwire};
\node at (2,0) {\are};
\node at (3,0) {\cross};
\node at (4,0) {\cross};

\node at (0,1) {\vwire};
\node at (1,1) {\are};
\node at (2,1) {\jay};
\node at (3,1) {\vwire};
\node at (4,1) {\are};

\node at (0,2) {\are};
\node at (1,2) {\jay};
\node at (2,2) {\vwire};
\node at (3,2) {\are};
\node at (4,2) {\hwire};

\node at (0,3) {\nowire};
\node at (1,3) {\are};
\node at (2,3) {\cross};
\node at (3,3) {\hwire};
\node at (4,3) {\hwire};

\node at (0,4) {\nowire};
\node at (1,4) {\nowire};
\node at (2,4) {\are};
\node at (3,4) {\hwire};
\node at (4,4) {\hwire};

\draw[thick] (-0.5,-0.5) -- (-0.5,4.5) -- (4.5,4.5) -- (4.5,-0.5) -- (-0.5,-0.5);

\node at (0,-1) {1};
\node at (1,-1) {2};
\node at (2,-1) {3};
\node at (3,-1) {4};
\node at (4,-1) {5};

\node at (5,4) {2};
\node at (5,3) {1};
\node at (5,2) {4};
\node at (5,1) {5};
\node at (5,0) {3};

\node at (2,-2.5) {Example};
\end{tikzpicture} \; \hspace{0.5in} \;
\begin{tikzpicture}[scale=0.4]
\node at (0,0) {\vwire};
\node at (1,0) {\vwire};
\node at (2,0) {\are};
\node at (3,0) {\cross};
\node at (4,0) {\cross};

\node at (0,1) {\are};
\node at (1,1) {\cross};
\draw[red,thick] (1,1) circle (0.4);
\node at (2,1) {\jay};
\node at (3,1) {\vwire};
\node at (4,1) {\are};

\node at (0,2) {\nowire};
\node at (1,2) {\vwire};
\node at (2,2) {\vwire};
\node at (3,2) {\are};
\node at (4,2) {\hwire};

\node at (0,3) {\nowire};
\node at (1,3) {\are};
\node at (2,3) {\cross};
\draw[red,thick] (2,3) circle (0.4);
\node at (3,3) {\hwire};
\node at (4,3) {\hwire};

\node at (0,4) {\nowire};
\node at (1,4) {\nowire};
\node at (2,4) {\are};
\node at (3,4) {\hwire};
\node at (4,4) {\hwire};

\draw[thick] (-0.5,-0.5) -- (-0.5,4.5) -- (4.5,4.5) -- (4.5,-0.5) -- (-0.5,-0.5);

\node at (0,-1) {1};
\node at (1,-1) {2};
\node at (2,-1) {3};
\node at (3,-1) {4};
\node at (4,-1) {5};

\node at (5,4) {1};
\node at (5,3) {2};
\node at (5,2) {4};
\node at (5,1) {5};
\node at (5,0) {3};

\node at (2,-2.5) {Non-example};
\end{tikzpicture} 
\]

The pipes are labeled $1$ to $n$ from left-to-right across the bottom, and the order in which they exit from top-to-bottom determines the \emph{permutation} of a bumpless pipe dream.
Given $w \in S_\infty$, let $\text{BPD}(w)$ denote the set of all bumpless pipe dreams with permutation $w$.
Below are the elements of $\text{BPD}(1432)$:
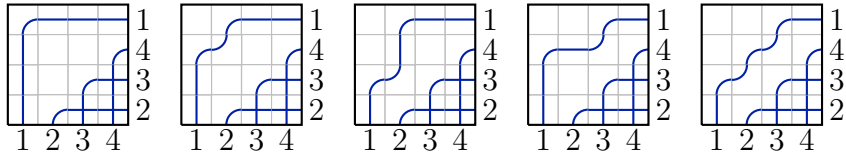
\[
\begin{tikzpicture}[scale=0.4]
\node at (0,0) {\vwire};
\node at (1,0) {\are};
\node at (2,0) {\cross};
\node at (3,0) {\cross};

\node at (0,1) {\vwire};
\node at (1,1) {\nowire};
\node at (2,1) {\are};
\node at (3,1) {\cross};

\node at (0,2) {\vwire};
\node at (1,2) {\nowire};
\node at (2,2) {\nowire};
\node at (3,2) {\are};

\node at (0,3) {\are};
\node at (1,3) {\hwire};
\node at (2,3) {\hwire};
\node at (3,3) {\hwire};

\draw[thick] (-0.5,-0.5) -- (-0.5,3.5) -- (3.5,3.5) -- (3.5,-0.5) -- (-0.5,-0.5);

\node at (0,-1) {1};
\node at (1,-1) {2};
\node at (2,-1) {3};
\node at (3,-1) {4};

\node at (4,3) {1};
\node at (4,2) {4};
\node at (4,1) {3};
\node at (4,0) {2};

\end{tikzpicture} \; 
\begin{tikzpicture}[scale=0.4]
\node at (0,0) {\vwire};
\node at (1,0) {\are};
\node at (2,0) {\cross};
\node at (3,0) {\cross};

\node at (0,1) {\vwire};
\node at (1,1) {\nowire};
\node at (2,1) {\are};
\node at (3,1) {\cross};

\node at (0,2) {\are};
\node at (1,2) {\jay};
\node at (2,2) {\nowire};
\node at (3,2) {\are};

\node at (0,3) {\nowire};
\node at (1,3) {\are};
\node at (2,3) {\hwire};
\node at (3,3) {\hwire};

\draw[thick] (-0.5,-0.5) -- (-0.5,3.5) -- (3.5,3.5) -- (3.5,-0.5) -- (-0.5,-0.5);

\node at (0,-1) {1};
\node at (1,-1) {2};
\node at (2,-1) {3};
\node at (3,-1) {4};

\node at (4,3) {1};
\node at (4,2) {4};
\node at (4,1) {3};
\node at (4,0) {2};

\end{tikzpicture} \; 
\begin{tikzpicture}[scale=0.4]
\node at (0,0) {\vwire};
\node at (1,0) {\are};
\node at (2,0) {\cross};
\node at (3,0) {\cross};

\node at (0,1) {\are};
\node at (1,1) {\jay};
\node at (2,1) {\are};
\node at (3,1) {\cross};

\node at (0,2) {\nowire};
\node at (1,2) {\vwire};
\node at (2,2) {\nowire};
\node at (3,2) {\are};

\node at (0,3) {\nowire};
\node at (1,3) {\are};
\node at (2,3) {\hwire};
\node at (3,3) {\hwire};

\draw[thick] (-0.5,-0.5) -- (-0.5,3.5) -- (3.5,3.5) -- (3.5,-0.5) -- (-0.5,-0.5);

\node at (0,-1) {1};
\node at (1,-1) {2};
\node at (2,-1) {3};
\node at (3,-1) {4};

\node at (4,3) {1};
\node at (4,2) {4};
\node at (4,1) {3};
\node at (4,0) {2};

\end{tikzpicture} \;
\begin{tikzpicture}[scale=0.4]
\node at (0,0) {\vwire};
\node at (1,0) {\are};
\node at (2,0) {\cross};
\node at (3,0) {\cross};

\node at (0,1) {\vwire};
\node at (1,1) {\nowire};
\node at (2,1) {\are};
\node at (3,1) {\cross};

\node at (0,2) {\are};
\node at (1,2) {\hwire};
\node at (2,2) {\jay};
\node at (3,2) {\are};

\node at (0,3) {\nowire};
\node at (1,3) {\nowire};
\node at (2,3) {\are};
\node at (3,3) {\hwire};

\draw[thick] (-0.5,-0.5) -- (-0.5,3.5) -- (3.5,3.5) -- (3.5,-0.5) -- (-0.5,-0.5);

\node at (0,-1) {1};
\node at (1,-1) {2};
\node at (2,-1) {3};
\node at (3,-1) {4};

\node at (4,3) {1};
\node at (4,2) {4};
\node at (4,1) {3};
\node at (4,0) {2};

\end{tikzpicture} \; 
\begin{tikzpicture}[scale=0.4]
\node at (0,0) {\vwire};
\node at (1,0) {\are};
\node at (2,0) {\cross};
\node at (3,0) {\cross};

\node at (0,1) {\are};
\node at (1,1) {\jay};
\node at (2,1) {\are};
\node at (3,1) {\cross};

\node at (0,2) {\nowire};
\node at (1,2) {\are};
\node at (2,2) {\jay};
\node at (3,2) {\are};

\node at (0,3) {\nowire};
\node at (1,3) {\nowire};
\node at (2,3) {\are};
\node at (3,3) {\hwire};

\draw[thick] (-0.5,-0.5) -- (-0.5,3.5) -- (3.5,3.5) -- (3.5,-0.5) -- (-0.5,-0.5);

\node at (0,-1) {1};
\node at (1,-1) {2};
\node at (2,-1) {3};
\node at (3,-1) {4};

\node at (4,3) {1};
\node at (4,2) {4};
\node at (4,1) {3};
\node at (4,0) {2};

\end{tikzpicture}
\]

For $v$ vexillary, Weigandt  recently introduced a bijection~\cite{weigandt2021bumpless} 
\[
\gamma: \text{BPD}(v) \rightarrow \text{SSYT}_{\phi(v)}(\lambda(v)).
\]
The correspondence can be described as follows.
Given $B \in \text{BPD}(v)$, fill each $\inlinenowire$-tile with the number of pipes that are above it.
Next, ignore all pipes and slide each number-filled tile northwest along its main diagonal to create a partition in the upper-left corner.
This step is well-defined as $v$ is vexillary.
Lastly, increase the value of each filling by its row index, i.e., add $i$ to each filling in row $i$ for all rows.
For example, with $B \in \text{BPD}(12587634)$ as below, we compute $\gamma(B)$:
\[
\centering
\begin{tabular}{ccc}
\begin{tikzpicture}[scale=0.4]
\node at (-2,4) {$B = $};
\node at (0,0) {\vwire};
\node at (1,0) {\vwire};
\node at (2,0) {\vwire};
\node at (3,0) {\are};
\node at (4,0) {\cross};
\node at (5,0) {\cross};
\node at (6,0) {\cross};
\node at (7,0) {\cross};

\node at (0,1) {\vwire};
\node at (1,1) {\vwire};
\node at (2,1) {\are};
\node at (3,1) {\hwire};
\node at (4,1) {\cross};
\node at (5,1) {\cross};
\node at (6,1) {\cross};
\node at (7,1) {\cross};

\node at (0,2) {\vwire};
\node at (1,2) {\vwire};
\node at (2,2) {\nowire};
\node at (2,2) {2};
\node at (3,2) {\nowire};
\node at (3,2) {2};
\node at (4,2) {\vwire};
\node at (5,2) {\are};
\node at (6,2) {\cross};
\node at (7,2) {\cross};

\node at (0,3) {\vwire};
\node at (1,3) {\vwire};
\node at (2,3) {\nowire};
\node at (2,3) {2};
\node at (3,3) {\nowire};
\node at (3,3) {2};
\node at (4,3) {\are};
\node at (5,3) {\jay};
\node at (6,3) {\are};
\node at (7,3) {\cross};

\node at (0,4) {\vwire};
\node at (1,4) {\are};
\node at (2,4) {\hwire};
\node at (3,4) {\hwire};
\node at (4,4) {\jay};
\node at (5,4) {\are};
\node at (6,4) {\jay};
\node at (7,4) {\are};

\node at (0,5) {\vwire};
\node at (1,5) {\nowire};
\node at (1,5) {1};
\node at (2,5) {\nowire};
\node at (2,5) {1};
\node at (3,5) {\nowire};
\node at (3,5) {1};
\node at (4,5) {\vwire};
\node at (5,5) {\nowire};
\node at (5,5) {2};
\node at (6,5) {\are};
\node at (7,5) {\hwire};

\node at (0,6) {\are};
\node at (1,6) {\hwire};
\node at (2,6) {\jay};
\node at (3,6) {\nowire};
\node at (3,6) {1};
\node at (4,6) {\are};
\node at (5,6) {\hwire};
\node at (6,6) {\hwire};
\node at (7,6) {\hwire};

\node at (0,7) {\nowire};
\node at (0,7) {0};
\node at (1,7) {\nowire};
\node at (1,7) {0};
\node at (2,7) {\are};
\node at (3,7) {\hwire};
\node at (4,7) {\hwire};
\node at (5,7) {\hwire};
\node at (6,7) {\hwire};
\node at (7,7) {\hwire};

\draw (-0.5,-0.5) -- (-0.5,7.5) -- (7.5,7.5) -- (7.5,-0.5) -- (-0.5,-0.5);

\node at (0,-1) {1};
\node at (1,-1) {2};
\node at (2,-1) {3};
\node at (3,-1) {4};
\node at (4,-1) {5};
\node at (5,-1) {6};
\node at (6,-1) {7};
\node at (7,-1) {8};

\node at (8,7) {1};
\node at (8,6) {2};
\node at (8,5) {5};
\node at (8,4) {8};
\node at (8,3) {7};
\node at (8,2) {6};
\node at (8,1) {3};
\node at (8,0) {4};

\end{tikzpicture} &
\begin{tikzpicture}[scale=0.4]
\node at (0,0) {\nowire};
\node at (1,0) {\nowire};
\node at (2,0) {\nowire};
\node at (3,0) {\nowire};
\node at (4,0) {\nowire};
\node at (5,0) {\nowire};
\node at (6,0) {\nowire};
\node at (7,0) {\nowire};

\node at (0,1) {\nowire};
\node at (1,1) {\nowire};
\node at (2,1) {\nowire};
\node at (3,1) {\nowire};
\node at (4,1) {\nowire};
\node at (5,1) {\nowire};
\node at (6,1) {\nowire};
\node at (7,1) {\nowire};

\node at (0,2) {\nowire};
\node at (1,2) {\nowire};
\node at (2,2) {\nowire};
\node at (2,2) {2};
\node at (3,2) {\nowire};
\node at (3,2) {2};
\node at (4,2) {\nowire};
\node at (5,2) {\nowire};
\node at (6,2) {\nowire};
\node at (7,2) {\nowire};

\node at (0,3) {\nowire};
\node at (1,3) {\nowire};
\node at (2,3) {\nowire};
\node at (2,3) {2};
\node at (3,3) {\nowire};
\node at (3,3) {2};
\node at (4,3) {\nowire};
\node at (5,3) {\nowire};
\node at (6,3) {\nowire};
\node at (7,3) {\nowire};

\node at (0,4) {\nowire};
\node at (1,4) {\nowire};
\node at (2,4) {\nowire};
\node at (3,4) {\nowire};
\node at (4,4) {\nowire};
\node at (5,4) {\nowire};
\node at (6,4) {\nowire};
\node at (7,4) {\nowire};

\node at (0,5) {\nowire};
\node at (1,5) {\nowire};
\node at (1,5) {1};
\node at (2,5) {\nowire};
\node at (2,5) {1};
\node at (3,5) {\nowire};
\node at (3,5) {1};
\node at (4,5) {\nowire};
\node at (5,5) {\nowire};
\node at (5,5) {2};
\node at (6,5) {\nowire};
\node at (7,5) {\nowire};

\node at (0,6) {\nowire};
\node at (1,6) {\nowire};
\node at (2,6) {\nowire};
\node at (3,6) {\nowire};
\node at (3,6) {1};
\node at (4,6) {\nowire};
\node at (5,6) {\nowire};
\node at (6,6) {\nowire};
\node at (7,6) {\nowire};

\node at (0,7) {\nowire};
\node at (0,7) {0};
\node at (1,7) {\nowire};
\node at (1,7) {0};
\node at (2,7) {\nowire};
\node at (3,7) {\nowire};
\node at (4,7) {\nowire};
\node at (5,7) {\nowire};
\node at (6,7) {\nowire};
\node at (7,7) {\nowire};

\node at (0,-1) {};
\node at (1,-1) {};
\node at (2,-1) {};
\node at (3,-1) {};
\node at (4,-1) {};
\node at (5,-1) {};
\node at (6,-1) {};
\node at (7,-1) {};

\node at (8,7) {};
\node at (8,6) {};
\node at (8,5) {};
\node at (8,4) {};
\node at (8,3) {};
\node at (8,2) {};
\node at (8,1) {};
\node at (8,0) {};

\draw[dashed,->] (4.6,5.4) -- (3,7);
\draw[dashed,->] (2.6,6.4) -- (2,7);
\draw[dashed,->] (2.6,5.4) -- (2,6);
\draw[dashed,->] (1.6,5.4) -- (1,6);
\draw[dashed,->] (0.6,5.4) -- (0,6);
\draw[dashed,->] (1.6,3.4) -- (0,5);
\draw[dashed,->] (1.6,2.4) -- (0,4);
\draw[dashed,->] (2.6,3.4) -- (1.6,4.4);

\draw (-0.5,-0.5) -- (-0.5,7.5) -- (7.5,7.5) -- (7.5,-0.5) -- (-0.5,-0.5);

\end{tikzpicture} &
\begin{tikzpicture}[scale=0.4]
\node at (0,4) {\nowire};
\node at (0,4) {6};
\node at (1,4) {\nowire};
\node at (1,4) {6};

\node at (0,5) {\nowire};
\node at (0,5) {5};
\node at (1,5) {\nowire};
\node at (1,5) {5};

\node at (0,6) {\nowire};
\node at (0,6) {3};
\node at (1,6) {\nowire};
\node at (1,6) {3};
\node at (2,6) {\nowire};
\node at (2,6) {3};

\node at (0,7) {\nowire};
\node at (0,7) {1};
\node at (1,7) {\nowire};
\node at (1,7) {1};
\node at (2,7) {\nowire};
\node at (2,7) {2};
\node at (3,7) {\nowire};
\node at (3,7) {3};

\node at (0,-1) {};
\node at (1,-1) {};
\node at (2,-1) {};
\node at (3,-1) {};
\node at (4,-1) {};
\node at (5,-1) {};
\node at (6,-1) {};
\node at (7,-1) {};

\node at (8,7) {};
\node at (8,6) {};
\node at (8,5) {};
\node at (8,4) {};
\node at (8,3) {};
\node at (8,2) {};
\node at (8,1) {};
\node at (8,0) {};
\node at (4,5) {$=\gamma(B)$};
\end{tikzpicture}
\end{tabular}
\]

\section{Maps}

\subsection{Edelman-Greene insertion}  In~\cite{edelman1987balanced}, Edelman and Greene introduced  a variant of the Robinson-Schensted-Knuth (RSK) algorithm that associates tableaux to reduced words.
We work with a slightly modified version mapping each reduced compatible sequence to a pair of tableaux $(\widetilde{P},\widetilde{Q})$.
Our version generalizes column RSK.
The only modification is that when $x$ is inserted into a column already containing $x$, that column is unchanged and $x{+}1$ is inserted into the next column.
See Section~\ref{ss:EG-A} for a complete description.
We use column insertion to ensure the resulting $\widetilde{Q}$-tableau is semi-standard.

Edelman-Greene and RSK are the same when applied to reduced words for Grassmannian permutations since Definition~\ref{d:eg-insert} 2(a) is never required.
Therefore, properties of RSK extend to Edelman-Greene insertion for such reduced compatible sequences.
\begin{proposition}[{see \cite[Prop. 3.9.3]{sagan2001symmetric}}] \label{p:Sagan}
If $\binom{r_1, r_2, \ldots, r_\ell}{a_1, a_2, \ldots, a_\ell}$ is a reduced compatible sequence for a Grassmannian permutation then
\[
\widetilde{Q} \left( \binom{r_2, \ldots, r_\ell}{a_2, \ldots, a_{\ell}} \right) = \text{jdt} \left( \widetilde{Q} \left( \binom{r_1, r_2, \ldots, r_\ell}{a_1, a_2, \ldots, a_{\ell}} \right) \right).
\]
\end{proposition}

\subsection{The Gao-Huang Bijection}

Recently, Gao and Huang  gave a bijection $\varphi: \text{BPD}(w) \rightarrow \text{PD}(w)$ preserving several important properties~\cite{gao2021canonical}.
Their map is iterative, with each step performed by an operator $\nabla$ acting on bumpless pipe dreams.
We define $\nabla$ as Definition~\ref{d:nabla}.
For $w \in S_\infty$ and $B \in \text{BPD}(w)$, we have $\nabla(B) \in \text{BPD}(v)$ for some $v$ satisfying $\ell(v) = \ell(w) - 1$.
Using $\nabla$, one can construct a pair $\text{pop}(B) = \binom{r}{a}$.

\begin{definition}
Given $B$ in $\text{BPD}(w)$ with $\ell(w) = \ell$ then
\[
\varphi(B) = \binom{r_1, \ldots, r_\ell}{a_1, \ldots, a_\ell}
\]
where $\binom{r_i}{a_i} = \text{pop}(\nabla^{i-1} B)$ for $1 \leq i \leq \ell$.
\end{definition}

See Figure~\ref{fig:gao-huang} in Appendix~\ref{ss:gao-huang-A} for an example of this map.
By definition, we have the following relationship between $\varphi(B)$ and $\varphi(\nabla B)$.

\begin{lemma} \label{l:PhiNab}
If $\varphi(B) = \binom{r_1, r_2, \ldots, r_\ell}{a_1, a_2, \ldots, a_\ell}$ then $\varphi(\nabla B) = \binom{r_2, \ldots, r_\ell}{a_2, \ldots, a_\ell}$.	
\end{lemma}

In~\cite{huang2021schubert}, an operator equivalent to $\nabla$ was defined for Grassmannian bumpless pipe dreams, where the following property was observed:
\begin{proposition}[{\cite[Prop. 5.1]{huang2021schubert}}]\label{p:HuangCor}
Let $w$ be a Grassmannian permutation and $B \in \text{BPD}(w)$.
Then
\[
\text{jdt}(\gamma(B)) = \gamma(\nabla B).
\]
\end{proposition}

See Figure~\ref{fig:HuangCor} after Appendix~\ref{ss:gao-huang-A} for an example of this proposition.

\subsection{Little bumps} 
\label{ss:little}
Monk's rule is a recurrence for Schubert polynomials.
Little bumps are bijections on certain sets of  reduced words~\cite{little2003combinatorial} that extend\footnote{This extension was first described in~\cite{knutsonschubert}, where Knutson attributes it to Buch.} to bijections on reduced compatible sequences~\cite{billey2019bijective}, giving a bijective proof of Monk's formula.

For $w \in S_\infty$ and $(i,j)$ such that $\ell(w t_{ij}) = \ell(w) - 1$, the \emph{Little bump} $L_{ij}$ sends $\binom{\textbf{r}}{\textbf{a}} \in \text{PD}(w)$ to a reduced compatible sequence $L_{ij}\binom{\textbf{r}}{\textbf{a}} = \binom{\textbf{r}}{\textbf{b}}$  for some other permutation $v$.
A precise definition of $L_{ij}$ appears as Definition~\ref{d:little}.
We require two key properties of the Little bijection.
They are the adaptation of~\cite[Property~(2)]{little2003combinatorial} and~\cite[Prop.~1]{hamaker2014relating}, respectively, to our conventions.
Proofs appear in Section~\ref{ss:little-A}.

\begin{proposition}
\label{p:Little-Grassmannian}
	Let $w \in S_\infty$.
	For every $\binom{\textbf{r}}{\textbf{a}} \in \text{PD}(w)$, there exists a Grassmannian permutation $v$ and  a sequence $L_{i_1j_1}, \dots, L_{i_kj_k}$ of Little bumps so that
	\[
	L_{i_kj_k} \circ \dots \circ L_{i_1j_1} \binom{\textbf{r}}{\textbf{a}}\in \text{PD}(v).
	\]
\end{proposition}

\begin{proposition}
\label{p:hamaker-young}
	Let $w \in S_\infty$ and $(i,j)$ so that $\ell(wt_{ij}) = \ell(w) - 1$.
	For any $\binom{\textbf{r}}{\textbf{a}} \in \text{PD}(w)$, we have
	\[
	\widetilde{Q}\left(\binom{\textbf{r}}{\textbf{a}}\right) = \widetilde{Q}\left(L_{ij}\binom{\textbf{r}}{\textbf{a}}\right)
	\]
\end{proposition}

\subsection{Huang bumps}
In \cite{huang2020bijective} two maps were defined on bumpless pipe dreams to give a combinatorial realization of Monk's formula.
We need one of these maps, which we refer to as a Huang bump in analogy with Little bumps.
A precise definition of Huang bumps appears as Definition~\ref{d:huang}.
Huang bumps and Little bumps are directly related by the Gao-Huang bijection.

\begin{theorem}[{\cite[Thm.~4.5]{gao2021canonical}}]
\label{t:canonical}
Let $w \in S_\infty$ with $\ell(w t_{ij}) = \ell(w) - 1$.
For $B \in \text{BPD}(w)$, we have
\[
(L_{ij} \circ \varphi)(B) = (\varphi \circ H_{ij})(B).
\]
\end{theorem}
\noindent Below is an example for $B \in \text{BPD}(1432)$:
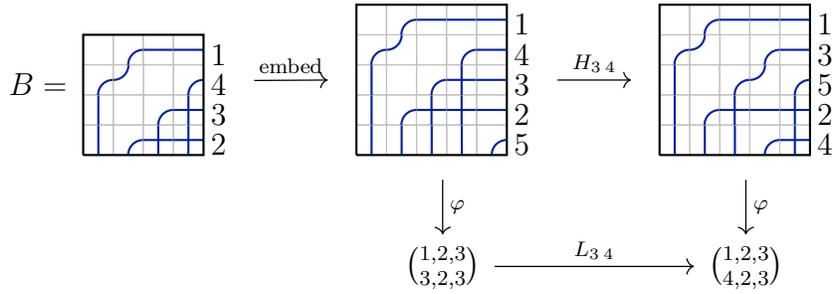
\[
\begin{tikzcd}
\begin{tikzpicture}[scale=0.4]
\node at (0,0) {\vwire};
\node at (1,0) {\are};
\node at (2,0) {\cross};
\node at (3,0) {\cross};

\node at (0,1) {\vwire};
\node at (1,1) {\nowire};
\node at (2,1) {\are};
\node at (3,1) {\cross};

\node at (0,2) {\are};
\node at (1,2) {\jay};
\node at (2,2) {\nowire};
\node at (3,2) {\are};

\node at (0,3) {\nowire};
\node at (1,3) {\are};
\node at (2,3) {\hwire};
\node at (3,3) {\hwire};

\draw[thick] (-0.5,0) -- (-0.5,4) -- (3.5,4) -- (3.5,0) -- (-0.5,0);

\node at (4,3) {1};
\node at (4,2) {4};
\node at (4,1) {3};
\node at (4,0) {2};
\node at (-2,2) {$B = $};
\end{tikzpicture} \arrow[start anchor=26.8, end anchor= real west]{r}{\text{embed}} &
\begin{tikzpicture}[scale=0.4]
\node at (0,0) {\vwire};
\node at (1,0) {\vwire};
\node at (2,0) {\vwire};
\node at (3,0) {\vwire};
\node at (4,0) {\are};

\node at (0,1) {\vwire};
\node at (1,1) {\are};
\node at (2,1) {\cross};
\node at (3,1) {\cross};
\node at (4,1) {\hwire};

\node at (0,2) {\vwire};
\node at (1,2) {\nowire};
\node at (2,2) {\are};
\node at (3,2) {\cross};
\node at (4,2) {\hwire};

\node at (0,3) {\are};
\node at (1,3) {\jay};
\node at (2,3) {\nowire};
\node at (3,3) {\are};
\node at (4,3) {\hwire};

\node at (0,4) {\nowire};
\node at (1,4) {\are};
\node at (2,4) {\hwire};
\node at (3,4) {\hwire};
\node at (4,4) {\hwire};

\draw[thick] (-0.5,0) -- (-0.5,5) -- (4.5,5) -- (4.5,0) -- (-0.5,0);

\node at (5,4) {1};
\node at (5,3) {4};
\node at (5,2) {3};
\node at (5,1) {2};
\node at (5,0) {5};
\end{tikzpicture} \arrow[start anchor = real east, end anchor = real west]{r}{H_{3 \: 4}} \arrow{d}{\varphi} &
\begin{tikzpicture}[scale=0.4]
\node at (0,0) {\vwire};
\node at (1,0) {\vwire};
\node at (2,0) {\vwire};
\node at (3,0) {\are};
\node at (4,0) {\cross};

\node at (0,1) {\vwire};
\node at (1,1) {\are};
\node at (2,1) {\cross};
\node at (3,1) {\hwire};
\node at (4,1) {\cross};

\node at (0,2) {\vwire};
\node at (1,2) {\nowire};
\node at (2,2) {\are};
\node at (3,2) {\jay};
\node at (4,2) {\are};

\node at (0,3) {\are};
\node at (1,3) {\jay};
\node at (2,3) {\nowire};
\node at (3,3) {\are};
\node at (4,3) {\hwire};

\node at (0,4) {\nowire};
\node at (1,4) {\are};
\node at (2,4) {\hwire};
\node at (3,4) {\hwire};
\node at (4,4) {\hwire};

\draw[thick] (-0.5,0) -- (-0.5,5) -- (4.5,5) -- (4.5,0) -- (-0.5,0);

\node at (5,4) {1};
\node at (5,3) {3};
\node at (5,2) {5};
\node at (5,1) {2};
\node at (5,0) {4};
\end{tikzpicture} \arrow{d}{\varphi} \\
 & 
\binom{1,2,3}{3,2,3} \arrow{r}{L_{3 \: 4}} &
\binom{1,2,3}{4,2,3}
\end{tikzcd} 
\]

We require one technical property of Huang bumps.
This follows from Lemma~\ref{l:huang-vexillary}, which states that Huang bumps for vexillary permutations do not alter the location of $\inlinenowire$-tiles.

\begin{corollary}
\label{c:huang-vexillary}
If $B$ and $B' = H_{ij}(B)$ are bumpless pipe dreams for vexillary permutations, then $\gamma(B) = \gamma(B')$. 
\end{corollary}


\section{Main Results}

We now state Lenart's bijection.

\begin{theorem}[{\cite[Rem.~4.12~(2)]{L04}}]
\label{t:lenart} 
For $v$ a vexillary permutation, the map $\widetilde{Q}: \text{PD}(v) \rightarrow \text{SSYT}_{\phi(v)}(\lambda(v))$ is a bijection.
\end{theorem}

In this section, we will prove Theorem~\ref{t:lenart} by showing the Gao-Huang bijection $\varphi: \text{BPD}(v) \rightarrow \text{PD}(v)$ preserves the underlying semistandard Young tableau associated to $B$ in $\text{BPD}(v)$ and $\varphi(B)$ in $\text{PD}(v)$. 

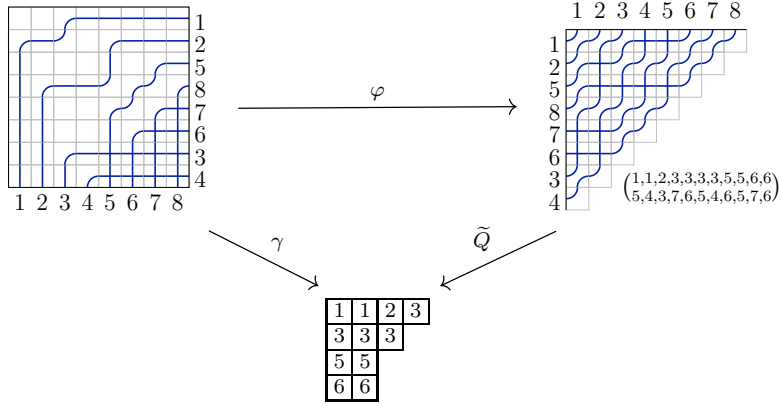
\begin{figure}[h!] 
\centering
\begin{tikzcd}[ampersand replacement=\&]
\scalebox{0.75}{
\begin{tikzpicture}[scale=0.4]
\node at (0,0) {\vwire};
\node at (1,0) {\vwire};
\node at (2,0) {\vwire};
\node at (3,0) {\are};
\node at (4,0) {\cross};
\node at (5,0) {\cross};
\node at (6,0) {\cross};
\node at (7,0) {\cross};

\node at (0,1) {\vwire};
\node at (1,1) {\vwire};
\node at (2,1) {\are};
\node at (3,1) {\hwire};
\node at (4,1) {\cross};
\node at (5,1) {\cross};
\node at (6,1) {\cross};
\node at (7,1) {\cross};

\node at (0,2) {\vwire};
\node at (1,2) {\vwire};
\node at (2,2) {\nowire};
\node at (3,2) {\nowire};
\node at (4,2) {\vwire};
\node at (5,2) {\are};
\node at (6,2) {\cross};
\node at (7,2) {\cross};

\node at (0,3) {\vwire};
\node at (1,3) {\vwire};
\node at (2,3) {\nowire};
\node at (3,3) {\nowire};
\node at (4,3) {\are};
\node at (5,3) {\jay};
\node at (6,3) {\are};
\node at (7,3) {\cross};

\node at (0,4) {\vwire};
\node at (1,4) {\are};
\node at (2,4) {\hwire};
\node at (3,4) {\hwire};
\node at (4,4) {\jay};
\node at (5,4) {\are};
\node at (6,4) {\jay};
\node at (7,4) {\are};

\node at (0,5) {\vwire};
\node at (1,5) {\nowire};
\node at (2,5) {\nowire};
\node at (3,5) {\nowire};
\node at (4,5) {\vwire};
\node at (5,5) {\nowire};
\node at (6,5) {\are};
\node at (7,5) {\hwire};

\node at (0,6) {\are};
\node at (1,6) {\hwire};
\node at (2,6) {\jay};
\node at (3,6) {\nowire};
\node at (4,6) {\are};
\node at (5,6) {\hwire};
\node at (6,6) {\hwire};
\node at (7,6) {\hwire};

\node at (0,7) {\nowire};
\node at (1,7) {\nowire};
\node at (2,7) {\are};
\node at (3,7) {\hwire};
\node at (4,7) {\hwire};
\node at (5,7) {\hwire};
\node at (6,7) {\hwire};
\node at (7,7) {\hwire};

\draw (-0.5,0) -- (-0.5,8) -- (7.5,8) -- (7.5,0) -- (-0.5,0);

\node at (0,-1) {1};
\node at (1,-1) {2};
\node at (2,-1) {3};
\node at (3,-1) {4};
\node at (4,-1) {5};
\node at (5,-1) {6};
\node at (6,-1) {7};
\node at (7,-1) {8};

\node at (8,7) {1};
\node at (8,6) {2};
\node at (8,5) {5};
\node at (8,4) {8};
\node at (8,3) {7};
\node at (8,2) {6};
\node at (8,1) {3};
\node at (8,0) {4};
\end{tikzpicture} } \arrow[start anchor=real east, end anchor=real west]{rr}{\varphi} \arrow{dr}{\gamma} \& \&
\scalebox{0.75}{
\begin{tikzpicture}[scale=0.4]
\node at (0,-1) {\jay};

\node at (0,0) {\bump};
\node at (1,0) {\jay};

\node at (0,1) {\cross};
\node at (1,1) {\cross};
\node at (2,1) {\jay};

\node at (0,2) {\cross};
\node at (1,2) {\cross};
\node at (2,2) {\bump};
\node at (3,2) {\jay};

\node at (0,3) {\bump};
\node at (1,3) {\bump};
\node at (2,3) {\bump};
\node at (3,3) {\bump};
\node at (4,3) {\jay};

\node at (0,4) {\bump};
\node at (1,4) {\cross};
\node at (2,4) {\cross};
\node at (3,4) {\cross};
\node at (4,4) {\cross};
\node at (5,4) {\jay};

\node at (0,5) {\bump};
\node at (1,5) {\cross};
\node at (2,5) {\bump};
\node at (3,5) {\bump};
\node at (4,5) {\bump};
\node at (5,5) {\bump};
\node at (6,5) {\jay};

\node at (0,6) {\bump};
\node at (1,6) {\bump};
\node at (2,6) {\bump};
\node at (3,6) {\cross};
\node at (4,6) {\cross};
\node at (5,6) {\bump};
\node at (6,6) {\bump};
\node at (7,6) {\jay};


\node at (0,7.5) {1};
\node at (1,7.5) {2};
\node at (2,7.5) {3};
\node at (3,7.5) {4};
\node at (4,7.5) {5};
\node at (5,7.5) {6};
\node at (6,7.5) {7};
\node at (7,7.5) {8};

\node at (-1,6) {1};
\node at (-1,5) {2};
\node at (-1,4) {5};
\node at (-1,3) {8};
\node at (-1,2) {7};
\node at (-1,1) {6};
\node at (-1,0) {3};
\node at (-1,-1) {4};

\draw (-0.5,-1) -- (-0.5,7) -- (7.5,7);

\node at (5.5,-0.25) {$\binom{1,1,2,3,3,3,3,5,5,6,6}{5,4,3,7,6,5,4,6,5,7,6}$};

\end{tikzpicture} } \arrow[swap]{dl}{\widetilde{Q}} \\ \&
\begin{ytableau}
1 & 1 & 2 & 3  \\
3 & 3 & 3 \\
5 & 5 \\
6 & 6
\end{ytableau} \&
\end{tikzcd}
\caption{An example of the maps.
\label{fig:bijections}}
\end{figure}

We first prove the statement for Grassmannian permutations.
Our argument uses the following observation.


\begin{lemma}
\label{l:grassmannian-induction}
For $w$ Grassmannian and $B \in \text{BPD}(w)$ one has
\[
\text{jdt}(\widetilde{Q}(\varphi(B))) = \widetilde{Q}(\varphi(\nabla B)).
\]
\end{lemma}

\begin{proof}

	The result follows by combining Proposition~\ref{p:Sagan} and Lemma~\ref{l:PhiNab} with Proposition~\ref{p:HuangCor}.
\end{proof}

\begin{theorem}
\label{t:Grassmannian}
For $w$ Grassmannian with descent in position $k$, the following diagram commutes:
\[
\begin{tikzcd}
\text{BPD}(w) \arrow{rr}{\varphi} \arrow{dr}{\gamma} & & \text{PD}(w) \arrow[swap]{dl}{\widetilde{Q}} \\ & \text{SSYT}_k(w) &
\end{tikzcd}
\]
\end{theorem}
\begin{proof}
Let $\ell(w) = p$ and $B \in \text{BPD}(w)$.
We will show $\gamma(B) = \widetilde{Q}(\varphi(B))$ using induction on $p$.
If $p = 0$ then all objects are empty and the result holds.
Now suppose $p \geq 1$ and that the statement holds for all Grassmannian $w'$ with $\ell(w') < p$.
Write $\text{pop}(B) = \binom{r}{a}$.
By definition $\nabla B \in \text{BPD}(s_a w)$.
Moreover, $w' = s_a w$  is Grassmannian with $\ell(w') = p - 1$. 
We now compute
\begin{align*}
\widetilde{Q}(\varphi(B)) &= \text{jdt}^{-1}(\widetilde{Q}(\varphi(\nabla B)))  \tag{Lemma~\ref{l:grassmannian-induction}} \\
	&= \text{jdt}^{-1}(\gamma({\nabla B})) \tag{Induction} \\
	&= \gamma(B) \tag{Proposition~\ref{p:HuangCor}}.
\end{align*}
This completes our proof.
\end{proof}

Next, we obtain our main result, extending Theorem~\ref{t:Grassmannian} to all vexillary permutations using Little bumps and Huang bumps.

\begin{theorem}
\label{t:main}
For $v$ vexillary, the following diagram commutes:
\[
\begin{tikzcd}
\text{BPD}(v) \arrow{rr}{\varphi} \arrow{dr}{\gamma} & & \text{PD}(v) \arrow[swap]{dl}{\widetilde{Q}} \\ & \text{SSYT}_{\phi(v)}(\lambda(v)) &
\end{tikzcd}
\]
\end{theorem}

\begin{proof}
Let $B \in \text{BPD}(v)$ and $P  = \varphi(B)$.
By Proposition~\ref{p:Little-Grassmannian}, there exists a sequence of Little bumps $L_{i_1j_1}, \dots, L_{i_kj_k}$ so that
\[
P' = L_{i_kj_k} \circ \dots \circ L_{i_1 j_1} P
\] is a reduced compatible sequence for a Grassmannian permutation $w$.
Set $B' = \varphi^{-1}(P')$.
By applying Theorem~\ref{t:canonical} repeatedly, we then have 
\[
B' = H_{i_kj_k} \circ \dots \circ H_{i_1 j_1} B.
\]
Then $\gamma(B) = \gamma(B') = \widetilde{Q}(P') = \widetilde{Q}(P) = \widetilde{Q}(\varphi(B))$, with the first three equalities following from
Corollary~\ref{c:huang-vexillary}, Theorem~\ref{t:Grassmannian} and Proposition~\ref{p:hamaker-young}, respectively and the fourth by construction.
\end{proof}

Since $\varphi$ and $\gamma$ are bijections, Theorem~\ref{t:lenart} is an immediate corollary.

\section{Huang bumps and the Little bijection}
\label{s:recording}

The Little bijection, as originally defined, maps a reduced word $\textbf{a}$ to a tableau $\text{LS}(\mathbf{a})$ and path $(w^1,\dots,w^k)$ of permutations in the Lascoux-Sch\"utzenberger tree.
To invert the map, one sends $\text{LS}(\textbf{a})$ to a Grassmannian reduced word with appropriate descent, then applies a Little bump for each edge in the path.
The main result of~\cite{hamaker2014relating} shows that (with appropriate conventions) $\text{LS}(\textbf{a}) = \widetilde{Q}(\textbf{a})$.
In principle, the data encoded by $(w^1,\dots,w^k)$ and $\widetilde{P}(\textbf{a})$ are equivalent, but a direct correspondence is not known.
One subtlety is that the Little bijection frequently requires working with permutations whose support contains negative values.

In Sections~\ref{ss:little} and~\ref{ss:little-A}, we present an extension of Little bumps to pipe dreams.
The Little bijection extends to biwords via the same filling procedure we give to construct $\widetilde{Q}$ in Section~\ref{ss:EG-A}.
However, the intermediate biwords constructed via the original map may violate the reduced compatible conditions. 
By instead applying the sequence of bumps constructed in Proposition~\ref{p:Little-Grassmannian}, one has a construction for pipe dreams in the spirit of the original Little bijection, which we call the \emph{Little bijection for pipe dreams}.

The canonical nature of the Gao-Huang bijection (Theorem~\ref{t:canonical}) allows one to work with Little bumps for pipe dreams and Huang bumps for bumpless pipe dreams interchangeably.
In particular, the pipe dream analogue of the Little bijection for pipe dreams extends to Huang bumps.

\begin{definition}
\label{d:huang-bijection}
	Let $w \in S_\infty$, $B \in \text{BPD}(w)$ and $P = \varphi(B)$.
	With $L_{i_1j_1} \dots L_{i_kj_k}$ the sequence of Little bumps for $P$ in Proposition~\ref{p:Little-Grassmannian}, the \emph{Huang bijection} maps $B$ to the pair
	\[
	(w^0, w^1, \dots, w^k) \quad \mbox{and} \quad \text{LS}(B) = \gamma \left( H_{i_kj_k} \circ \dots \circ H_{i_1 j_1} B\right)
	\]
	where $w^0 = w$ and  $w^\ell$ is the permutation of $H_{i_\ell j_\ell} \circ \dots \circ H_{i_1 j_1} B$ for $\ell > 0$.
\end{definition}

By repeated application of $\varphi$ and Theorem~\ref{t:canonical}, we see the Huang bijection is well-defined.
Moreover, it is invertible since $\gamma$ is invertible and each Huang bump is invertible if one knows the source permutation.

\begin{proposition}
\label{p:two-bijections}
	The Huang bijection and the Little bijection for pipe dreams coincide.
\end{proposition}

\begin{proof}
	This follows from the previous discussion and the fact that $\gamma$ coincides with usual correspondence for Grassmannian reduced words, which follows from Theorem~\ref{t:Grassmannian} and~\cite[Lem.~5]{hamaker2014relating}.
\end{proof}

The bulk of our proof for Theorem~\ref{t:main} is embedded in Proposition~\ref{p:two-bijections}.
The only missing piece is Corollary~\ref{c:recording-huang}, which equivalently shows we only need to apply Huang bumps in Definition~\ref{d:huang-bijection} until the corresponding permutation $w^\ell$ is vexillary.

For the vexillary case, we understand the range of these maps precisely: the path of permutations is deterministic and $\text{LS}$ maps to flagged tableaux.
Outside of this setting the set of paths can be determined, but the image of $\text{LS}$ does not have a simple description.
However, Proposition~\ref{p:hamaker-young} applies, giving:

\begin{corollary}
	\label{c:recording-huang}
	For $w \in S_\infty$ and $B \in \text{BPD}(w)$, we have
	\[
	\text{LS}(B) = \widetilde{Q}(\varphi(B)).
	\]
\end{corollary}

In principle, the sequence of permutations 
$
(w^0,\dots,w^k)
$
constructed in the Huang bijection encodes the same data as the Edelman-Greene insertion tableau $\widetilde{P}(\varphi(B))$.
However, there is no immediate correspondence.
There is another way to associate an insertion tableau to a bumpless pipe dream, obtained by composing the map in~\cite[Thm.~5.19]{lam2021back} with either the map $\Omega$ in~\cite[Thm.~1.7]{weigandt2021bumpless} or the map in~\cite[Thm.~6.1]{fan2018bumpless}, but it is unclear whether either is consistent with $\text{LS}$.
Forthcoming work of Huang, Shimozono and Yu gives another, more general, realization of Edelman-Greene insertion on bumpless pipe dreams~\cite{private}, so we choose not to pursue this direction further.

\begin{remark}
Recently, Huang and Pylyavskyy introduced two new bijections mapping bumpless pipe dreams to labeled chains of permutations~\cite{huang2022bumpless}.
Neither of their maps are precisely equivalent to the Huang bijection we describe, but they encode the same data.
Both maps are based on what we call inverse Huang bumps, and their right insertion is equivalent to the map in~\cite{billey2019bijective}.
One difference is that their chains involve permutations with different lengths, while the permutations in the Huang bijection all have the same length.

\end{remark}

\appendix

\section{Details of Maps}

\subsection{Edelman-Greene insertion}
\label{ss:EG-A}

\begin{definition}
\label{d:eg-insert}
Let $C = \{c_1 < \ldots < c_\ell\}$.
The \emph{Edelman-Greene insertion rule} for inserting a positive integer $x \neq c_\ell$ into $C$ is as follows:

\begin{itemize}
\item[1.)] If $C$ is empty or $x > c_{\ell}$  then append $x$ to the bottom of $C$.
\item[2.)] If $x < c_{\ell}$ let $k$ be the smallest row index for which $x < c_k$. 
\begin{itemize}
\item[(a)] If $c_k = x+1$ and $c_{k-1} = x$ then leave $C$ unchanged.
\item[(b)] Otherwise, replace $c_k$ with $x$. 
\end{itemize}
\end{itemize}
\end{definition}

Using the Edelman-Greene insertion rule, we explain how to insert an entry into a tableau.

\begin{definition}
	\label{d:eg-tableau}
Let $T$ be a semistandard tableau with columns $C_1, \ldots, C_m$ where $C_i = \{c_1^i < \ldots < c_{\ell_i}^i\}$.
To \emph{Edelman-Greene insert} $x$ into $T$, denoted $T \uparrow x$, iterate the following:
Using the Edelman-Greene insertion rule, insert $x$ into $C_1$.

\begin{itemize}
\item[1.)] If $x$ was appended to the bottom of $C_1$ then stop.
\item[2.)] If $C_1$ was left unchanged then insert $x+1$ into the tableau with columns $C_2, \ldots, C_m$. 
\item[3.)] If $x$ replaced $c_k^1$ in $C_1$ then insert $c_k^1$ into the tableau with columns $C_2, \ldots, C_m$. 
\end{itemize}
\end{definition}

Note this algorithm is not well defined for all integers and tableaux, since we could have $x = c_\ell$.
Item 2 is where Edelman-Greene insertion differs from RSK. 

\begin{definition}
\label{d:EG}
For a reduced compatible sequence $\binom{\mathbf{r}}{\mathbf{a}} = \binom{r_1, \ldots, r_\ell}{a_1, \ldots, a_\ell}$, use the Edelman-Greene algorithm to output a pair of tableaux $(\widetilde{P},\widetilde{Q})$ as follows. The $\widetilde{P}$-tableau is 
\[
\widetilde{P}\left( \binom{\mathbf{r}}{\mathbf{a}} \right) = (((\varnothing \uparrow a_1) \uparrow a_2) \ldots \uparrow a_\ell).
\]
Meanwhile, the $\widetilde{Q}$-tableau is constructed simultaneously by placing $r_i$ in the new box created when inserting $a_i$.
\end{definition}
It is a non-trivial fact that, when inserting a reduced word, every step is well-defined during this process~\cite{edelman1987balanced}.

For example, one has
\begin{align*}
\widetilde{P}\left( \binom{1,1,2,2}{4,2,3,2} \right) &: \quad 
\varnothing \quad 
\begin{ytableau}
4
\end{ytableau} \quad 
\begin{ytableau}
2 & 4
\end{ytableau} \quad 
\begin{ytableau}
2 & 4 \\ 3
\end{ytableau} \quad 
\begin{ytableau}
2 & 3 & 4 \\ 3
\end{ytableau} \\
\widetilde{Q}\left( \binom{1,1,2,2}{4,2,3,2} \right) &: \quad
\varnothing \quad 
\begin{ytableau}
1
\end{ytableau} \quad 
\begin{ytableau}
1 & 1
\end{ytableau} \quad 
\begin{ytableau}
1 & 1 \\ 2
\end{ytableau} \quad 
\begin{ytableau}
1 & 1 & 2 \\ 2
\end{ytableau} 
\end{align*}
where Definition~\ref{d:eg-insert} 2(a) was used when inserting the second $2$.

\subsection{The Gao-Huang bijection}
\label{ss:gao-huang-A}

Given a bumpless pipe dream $B$ write $B_{ij}$ for the tile of $B$ in location $(i,j)$. Let $B(\inlinenowire)$ denote the set of $(i,j)$ such that $B_{ij} = \inlinenowire$ and likewise for other tiles. 
We now recall \emph{droop moves} on a bumpless pipe dream.

\begin{definition}
For a bumpless pipe dream $B$, say there is an \emph{available droop} from $(i,j)$ into $(k,\ell)$ with $k > i$ and $\ell > j$ if 
\[
B_{ij} = \are, \quad B_{k \ell} = \nowire, \quad \text{and} \quad B_{pq} \neq \are, \jay
\]
for any $i \leq p \leq k$ and $j \leq q \leq \ell$. When available, a \emph{droop move} from $(i,j)$ into $(k,\ell)$ reroutes the pipe going through $(i,j)$ to instead go through $(k,\ell)$ as shown below. 
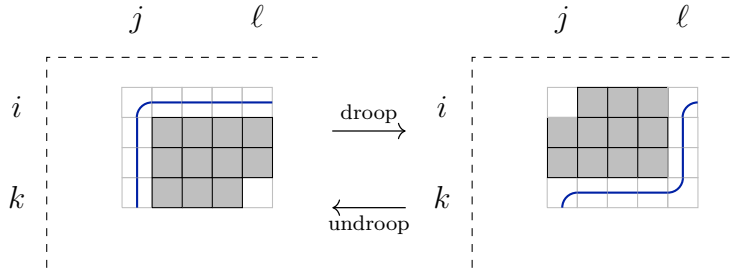
\[
\begin{tikzcd}
\begin{tikzpicture}[scale=0.4]
\node at (0,0) {$k$};
\node at (0,3) {$i$};
\node at (4,6) {$j$};
\node at (8,6) {$\ell$};
\node at (4,0) {\vwire};
\node at (4,1) {\vwire};
\node at (4,2) {\vwire};
\node at (4,3) {\are};
\node at (5,3) {\hwire};
\node at (6,3) {\hwire};
\node at (7,3) {\hwire};
\node at (8,3) {\hwire};
\node at (8,0) {\nowire};
\node at (7,0) {\shade};
\node at (6,0) {\shade};
\node at (5,0) {\shade};
\node at (5,1) {\shade};
\node at (6,1) {\shade};
\node at (7,1) {\shade};
\node at (8,1) {\shade};
\node at (8,2) {\shade};
\node at (7,2) {\shade};
\node at (6,2) {\shade};
\node at (5,2) {\shade};

\draw[dashed] (1,-2) -- (1,5) -- (10,5);
\end{tikzpicture} \arrow[start anchor = real east, end anchor = real west]{r}{\text{droop}} &
\begin{tikzpicture}[scale=0.4]
\node at (0,0) {$k$};
\node at (0,3) {$i$};
\node at (4,6) {$j$};
\node at (8,6) {$\ell$};
\node at (4,0) {\are};
\node at (4,1) {\shade};
\node at (4,2) {\shade};
\node at (4,3) {\nowire};
\node at (5,3) {\shade};
\node at (6,3) {\shade};
\node at (7,3) {\shade};
\node at (8,3) {\are};
\node at (8,0) {\jay};
\node at (7,0) {\hwire};
\node at (6,0) {\hwire};
\node at (5,0) {\hwire};
\node at (5,1) {\shade};
\node at (6,1) {\shade};
\node at (7,1) {\shade};
\node at (8,1) {\vwire};
\node at (8,2) {\vwire};
\node at (7,2) {\shade};
\node at (6,2) {\shade};
\node at (5,2) {\shade};

\draw[dashed] (1,-2) -- (1,5) -- (10,5); 
\end{tikzpicture}
\arrow[start anchor = base west, end anchor = base east]{l}{\text{undroop}}
\end{tikzcd}
\]
The $\inlineshade$-tiles must be from $\{\inlinenowire, \inlinevwire, \inlinehwire, \inlinecross\}$ and the $\inlinevwire$ and $\inlinehwire$ tiles might actually be $\inlinecross$ tiles in $B$.
Note that droop moves are invertible.
An \emph{undroop} from $(k,\ell)$ into $(i,j)$ is the inverse of the droop move from $(i,j)$ into $(k,\ell)$.
\end{definition}

We now give the definition of $\nabla$.

\begin{definition}
\label{d:nabla}
Given $B$ in $\text{BPD}(w)$, the following procedure records $\text{pop}(B) = \binom{r}{a}$ and creates $\nabla B$ in $\text{BPD}(s_a w)$.
To begin, let $r$ be the smallest row index in which $B$ has $\inlinenowire$-tiles. Mark the rightmost $\inlinenowire$-tile in row $r$ with an $\times$.
\begin{itemize}
\item[(1)] Move the mark $\times$ to the rightmost $\inlinenowire$-tile in a contiguous block of $\inlinenowire$-tiles in its row. Let $(i,j)$ be the location of the marked tile and let $p$ be the pipe going through $(i,j+1)$. 
\item[(2)] If $p \neq j+1$ then it must have a $\inlinejay$-tile at $(i',j+1)$ for some $i' > i$. Consider the $(i'-i) \times 2$ rectangle $R$ with northwest corner $(i,j)$ and southeast corner $(i',j+1)$. 
\begin{itemize}
\item[(a)] For any pipe $q$ crossing through $R$ having both a $\inlineare$-tile and a $\inlinejay$-tile in column $j$, say in rows $i < k < k' < i'$, droop the pipe $q$ at $(k,j)$ into $(k',j+1)$ whilst ignoring the pipe $p$. 
\item[(b)] Undroop pipe $p$ at $(i',j+1)$ into $(i,j)$. Move the mark $\times$ to $(i',j+1)$ and go back to Step (1).
\end{itemize}
\item[(3)] If $p = j+1$ then pipes $j$ and $j+1$ must intersect at some $(i',j+1)$ with $i' > i$. Replace this $\inlinecross$-tile with a $\inlinebump$-tile but then immediately undroop the $\inlinejay$-turn of pipe $j$ in this $\inlinebump$-tile into $(i,j)$, adjusting any intersecting pipes between rows $i$ and $i'$ in the same fashion as Step 2(a) above. Set $a = j$, the column index of the left column in this last $(i'-i) \times 2$ rectangle of tiles to be modified. The result is a bumpless pipe dream $\nabla B$ for $s_a w$.
\end{itemize}
The process ends after Step (3) and we output $\text{pop}(B) = \binom{r}{a}$ and $\nabla B$.
See Figure~\ref{fig:gao-huang} for examples of $\nabla$ and $\text{Pop}$.
\end{definition}

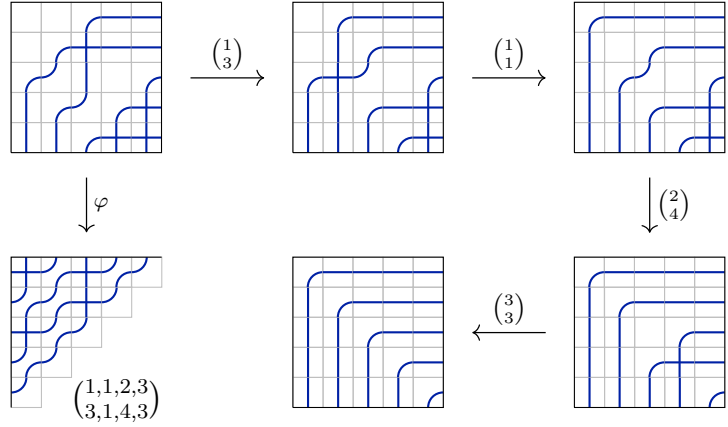
\begin{figure}
\scalebox{1}{
\begin{tikzcd}[ampersand replacement=\&]
\begin{tikzpicture}[scale=0.4]
\node at (0,0) {\vwire};
\node at (1,0) {\vwire};
\node at (2,0) {\are};
\node at (3,0) {\cross};
\node at (4,0) {\cross};

\node at (0,1) {\vwire};
\node at (1,1) {\are};
\node at (2,1) {\jay};
\node at (3,1) {\are};
\node at (4,1) {\cross};

\node at (0,2) {\are};
\node at (1,2) {\jay};
\node at (2,2) {\vwire};
\node at (3,2) {\nowire};
\node at (4,2) {\are};

\node at (0,3) {\nowire};
\node at (1,3) {\are};
\node at (2,3) {\cross};
\node at (3,3) {\hwire};
\node at (4,3) {\hwire};

\node at (0,4) {\nowire};
\node at (1,4) {\nowire};
\node at (2,4) {\are};
\node at (3,4) {\hwire};
\node at (4,4) {\hwire};

\draw (-0.5,0) -- (-0.5,5) -- (4.5,5) -- (4.5,0) -- (-0.5,0);
\end{tikzpicture} \arrow{d}{\varphi} \arrow[start anchor=real east, end anchor=real west]{r}{\binom{1}{3}} \&
\begin{tikzpicture}[scale=0.4]
\node at (0,0) {\vwire};
\node at (1,0) {\vwire};
\node at (2,0) {\vwire};
\node at (3,0) {\are};
\node at (4,0) {\cross};

\node at (0,1) {\vwire};
\node at (1,1) {\vwire};
\node at (2,1) {\are};
\node at (3,1) {\hwire};
\node at (4,1) {\cross};

\node at (0,2) {\are};
\node at (1,2) {\cross};
\node at (2,2) {\jay};
\node at (3,2) {\nowire};
\node at (4,2) {\are};

\node at (0,3) {\nowire};
\node at (1,3) {\vwire};
\node at (2,3) {\are};
\node at (3,3) {\hwire};
\node at (4,3) {\hwire};

\node at (0,4) {\nowire};
\node at (1,4) {\are};
\node at (2,4) {\hwire};
\node at (3,4) {\hwire};
\node at (4,4) {\hwire};

\draw (-0.5,0) -- (-0.5,5) -- (4.5,5) -- (4.5,0) -- (-0.5,0);
\end{tikzpicture} \arrow[start anchor=real east, end anchor=real west]{r}{\binom{1}{1}} \&
\begin{tikzpicture}[scale=0.4]
\node at (0,0) {\vwire};
\node at (1,0) {\vwire};
\node at (2,0) {\vwire};
\node at (3,0) {\are};
\node at (4,0) {\cross};

\node at (0,1) {\vwire};
\node at (1,1) {\vwire};
\node at (2,1) {\are};
\node at (3,1) {\hwire};
\node at (4,1) {\cross};

\node at (0,2) {\vwire};
\node at (1,2) {\are};
\node at (2,2) {\jay};
\node at (3,2) {\nowire};
\node at (4,2) {\are};

\node at (0,3) {\vwire};
\node at (1,3) {\nowire};
\node at (2,3) {\are};
\node at (3,3) {\hwire};
\node at (4,3) {\hwire};

\node at (0,4) {\are};
\node at (1,4) {\hwire};
\node at (2,4) {\hwire};
\node at (3,4) {\hwire};
\node at (4,4) {\hwire};

\draw (-0.5,0) -- (-0.5,5) -- (4.5,5) -- (4.5,0) -- (-0.5,0);
\end{tikzpicture} \arrow{d}{\binom{2}{4}} \\
\begin{tikzpicture}[scale=0.4]
\node at (0,0) {\jay};

\node at (0,1) {\bump};
\node at (1,1) {\jay};

\node at (0,2) {\cross};
\node at (1,2) {\bump};
\node at (2,2) {\jay};

\node at (0,3) {\bump};
\node at (1,3) {\bump};
\node at (2,3) {\cross};
\node at (3,3) {\jay};

\node at (0,4) {\cross};
\node at (1,4) {\bump};
\node at (2,4) {\cross};
\node at (3,4) {\bump};
\node at (4,4) {\jay};

\node at (3,0) {$\binom{1,1,2,3}{3,1,4,3}$};

\draw (-0.5,0) -- (-0.5,5) -- (4.5,5);
\end{tikzpicture} \&
\begin{tikzpicture}[scale=0.4]
\node at (0,0) {\vwire};
\node at (1,0) {\vwire};
\node at (2,0) {\vwire};
\node at (3,0) {\vwire};
\node at (4,0) {\are};

\node at (0,1) {\vwire};
\node at (1,1) {\vwire};
\node at (2,1) {\vwire};
\node at (3,1) {\are};
\node at (4,1) {\hwire};

\node at (0,2) {\vwire};
\node at (1,2) {\vwire};
\node at (2,2) {\are};
\node at (3,2) {\hwire};
\node at (4,2) {\hwire};

\node at (0,3) {\vwire};
\node at (1,3) {\are};
\node at (2,3) {\hwire};
\node at (3,3) {\hwire};
\node at (4,3) {\hwire};

\node at (0,4) {\are};
\node at (1,4) {\hwire};
\node at (2,4) {\hwire};
\node at (3,4) {\hwire};
\node at (4,4) {\hwire};

\draw (-0.5,0) -- (-0.5,5) -- (4.5,5) -- (4.5,0) -- (-0.5,0);
\end{tikzpicture} \& 
\begin{tikzpicture}[scale=0.4]
\node at (0,0) {\vwire};
\node at (1,0) {\vwire};
\node at (2,0) {\vwire};
\node at (3,0) {\vwire};
\node at (4,0) {\are};

\node at (0,1) {\vwire};
\node at (1,1) {\vwire};
\node at (2,1) {\are};
\node at (3,1) {\cross};
\node at (4,1) {\hwire};

\node at (0,2) {\vwire};
\node at (1,2) {\vwire};
\node at (2,2) {\nowire};
\node at (3,2) {\are};
\node at (4,2) {\hwire};

\node at (0,3) {\vwire};
\node at (1,3) {\are};
\node at (2,3) {\hwire};
\node at (3,3) {\hwire};
\node at (4,3) {\hwire};

\node at (0,4) {\are};
\node at (1,4) {\hwire};
\node at (2,4) {\hwire};
\node at (3,4) {\hwire};
\node at (4,4) {\hwire};

\draw (-0.5,0) -- (-0.5,5) -- (4.5,5) -- (4.5,0) -- (-0.5,0);
\end{tikzpicture} \arrow[start anchor=real west, end anchor=real east, swap]{l}{\binom{3}{3}}
\end{tikzcd}
}
\caption{The Gao-Huang bijection.
\label{fig:gao-huang}}	
\end{figure}

\begin{figure}[h!]
\centering
\begin{tikzcd}[ampersand replacement=\&]
\begin{tikzpicture}[scale=0.4]
\node at (0,0) {\vwire};
\node at (1,0) {\vwire};
\node at (2,0) {\vwire};
\node at (3,0) {\vwire};
\node at (4,0) {\are};
\node at (5,0) {\cross};
\node at (6,0) {\cross};

\node at (0,1) {\vwire};
\node at (1,1) {\vwire};
\node at (2,1) {\are};
\node at (3,1) {\cross};
\node at (4,1) {\hwire};
\node at (5,1) {\cross};
\node at (6,1) {\cross};

\node at (0,2) {\are};
\node at (1,2) {\cross};
\node at (2,2) {\hwire};
\node at (3,2) {\cross};
\node at (4,2) {\hwire};
\node at (5,2) {\cross};
\node at (6,2) {\cross};

\node at (0,3) {\nowire};
\node at (1,3) {\vwire};
\node at (2,3) {\nowire};
\node at (3,3) {\are};
\node at (4,3) {\jay};
\node at (5,3) {\vwire};
\node at (6,3) {\are};

\node at (0,4) {\nowire};
\node at (1,4) {\vwire};
\node at (2,4) {\nowire};
\node at (3,4) {\nowire};
\node at (4,4) {\vwire};
\node at (5,4) {\are};
\node at (6,4) {\hwire};

\node at (0,5) {\nowire};
\node at (1,5) {\are};
\node at (2,5) {\jay};
\node at (3,5) {\nowire};
\node at (4,5) {\are};
\node at (5,5) {\hwire};
\node at (6,5) {\hwire};

\node at (0,6) {\nowire};
\node at (1,6) {\nowire};
\node at (2,6) {\are};
\node at (3,6) {\hwire};
\node at (4,6) {\hwire};
\node at (5,6) {\hwire};
\node at (6,6) {\hwire};

\draw[thick] (-0.5,0) -- (-0.5,7) -- (6.5,7) -- (6.5,0) -- (-0.5,0);
\end{tikzpicture} \arrow[start anchor=real east, end anchor=real west]{r}{\binom{1}{5}} \arrow{d}{\gamma} \&
\begin{tikzpicture}[scale=0.4]
\node at (0,0) {\vwire};
\node at (1,0) {\vwire};
\node at (2,0) {\vwire};
\node at (3,0) {\vwire};
\node at (4,0) {\vwire};
\node at (5,0) {\are};
\node at (6,0) {\cross};

\node at (0,1) {\vwire};
\node at (1,1) {\vwire};
\node at (2,1) {\are};
\node at (3,1) {\cross};
\node at (4,1) {\cross};
\node at (5,1) {\hwire};
\node at (6,1) {\cross};

\node at (0,2) {\are};
\node at (1,2) {\cross};
\node at (2,2) {\hwire};
\node at (3,2) {\cross};
\node at (4,2) {\cross};
\node at (5,2) {\hwire};
\node at (6,2) {\cross};

\node at (0,3) {\nowire};
\node at (1,3) {\vwire};
\node at (2,3) {\nowire};
\node at (3,3) {\vwire};
\node at (4,3) {\are};
\node at (5,3) {\jay};
\node at (6,3) {\are};

\node at (0,4) {\nowire};
\node at (1,4) {\vwire};
\node at (2,4) {\nowire};
\node at (3,4) {\vwire};
\node at (4,4) {\nowire};
\node at (5,4) {\are};
\node at (6,4) {\hwire};

\node at (0,5) {\nowire};
\node at (1,5) {\vwire};
\node at (2,5) {\nowire};
\node at (3,5) {\are};
\node at (4,5) {\hwire};
\node at (5,5) {\hwire};
\node at (6,5) {\hwire};

\node at (0,6) {\nowire};
\node at (1,6) {\are};
\node at (2,6) {\hwire};
\node at (3,6) {\hwire};
\node at (4,6) {\hwire};
\node at (5,6) {\hwire};
\node at (6,6) {\hwire};

\draw[thick] (-0.5,0) -- (-0.5,7) -- (6.5,7) -- (6.5,0) -- (-0.5,0);
\end{tikzpicture} \arrow[start anchor=real east, end anchor=real west]{r}{\binom{1}{1}} \arrow{d}{\gamma} \&
\begin{tikzpicture}[scale=0.4]
\node at (0,0) {\vwire};
\node at (1,0) {\vwire};
\node at (2,0) {\vwire};
\node at (3,0) {\vwire};
\node at (4,0) {\vwire};
\node at (5,0) {\are};
\node at (6,0) {\cross};

\node at (0,1) {\vwire};
\node at (1,1) {\vwire};
\node at (2,1) {\are};
\node at (3,1) {\cross};
\node at (4,1) {\cross};
\node at (5,1) {\hwire};
\node at (6,1) {\cross};

\node at (0,2) {\vwire};
\node at (1,2) {\are};
\node at (2,2) {\hwire};
\node at (3,2) {\cross};
\node at (4,2) {\cross};
\node at (5,2) {\hwire};
\node at (6,2) {\cross};

\node at (0,3) {\vwire};
\node at (1,3) {\nowire};
\node at (2,3) {\nowire};
\node at (3,3) {\vwire};
\node at (4,3) {\are};
\node at (5,3) {\jay};
\node at (6,3) {\are};

\node at (0,4) {\vwire};
\node at (1,4) {\nowire};
\node at (2,4) {\nowire};
\node at (3,4) {\vwire};
\node at (4,4) {\nowire};
\node at (5,4) {\are};
\node at (6,4) {\hwire};

\node at (0,5) {\vwire};
\node at (1,5) {\nowire};
\node at (2,5) {\nowire};
\node at (3,5) {\are};
\node at (4,5) {\hwire};
\node at (5,5) {\hwire};
\node at (6,5) {\hwire};

\node at (0,6) {\are};
\node at (1,6) {\hwire};
\node at (2,6) {\hwire};
\node at (3,6) {\hwire};
\node at (4,6) {\hwire};
\node at (5,6) {\hwire};
\node at (6,6) {\hwire};

\draw[thick] (-0.5,0) -- (-0.5,7) -- (6.5,7) -- (6.5,0) -- (-0.5,0);
\end{tikzpicture} \arrow{d}{\gamma} \\
\begin{ytableau}
1&1&2 \\ 2 & 3 & 3 \\ 3 & 4 \\ 4
\end{ytableau} \arrow{r}{\text{jdt}(\gamma(B))} \&
\begin{ytableau}
1 & 2 & 3 \\ 2 & 3 \\ 3 & 4 \\ 4
\end{ytableau} \arrow{r}{\text{jdt}(\gamma(\nabla B))} \&
\begin{ytableau}
2 & 2 & 3 \\ 3 & 3 \\ 4 & 4
\end{ytableau} 
\end{tikzcd}
\caption{An illustration of Proposition~\ref{p:HuangCor}.\label{fig:HuangCor}}
\end{figure}
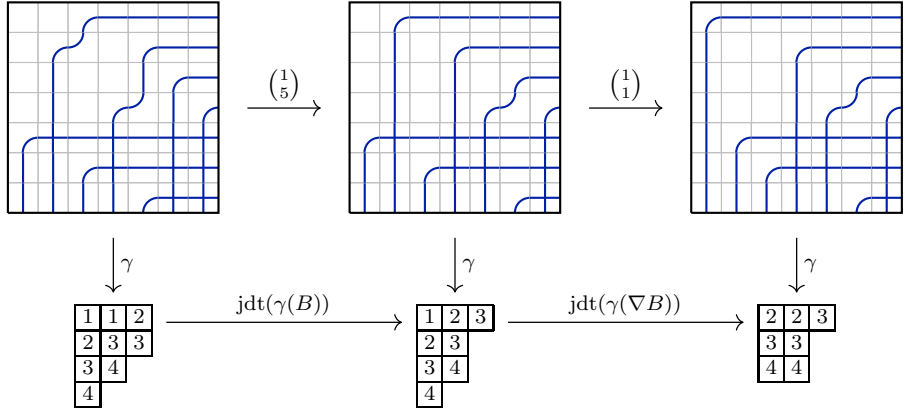

\subsection{Little bumps}
\label{ss:little-A}
We define Little bumps.
\begin{definition}
\label{d:little}
Let $w \in S_\infty$ and fix $(i,j)$ so that $\ell(w t_{ij}) = \ell(w) - 1$.
The \emph{(upward) Little bump} $L_{ij}$ takes as input some $\binom{\textbf{r}}{\textbf{a}} \in \text{PD}(w)$.
Fix $m$ so that wires $i$ and $j$ cross in the $\inlinecross$-tile corresponding to $\binom{r_{m}}{a_{m}}$.
\begin{enumerate}
	\item Set $\textbf{a}^{(1)} := \textbf{a}$, $m_1 := m$ and $k = 1$.
	\item Replace $\binom{r_{m_k}}{a_{m_k}}$ with $\binom{r_{m_k}}{a_{m_1}+k}$ to obtain a new biword $\binom{\textbf{r}}{\textbf{a}^{(k)}}$.
	\item If $\textbf{a}^{(k)}$ is a reduced word, return $\binom{\textbf{r}}{\textbf{a}^{(k)}}$.
	Otherwise, there exists a unique index $m_{k+1} \neq m_{k}$ so that the wires crossing in the $\inlinecross$-tile corresponding to $\binom{r_{m_k}}{a_{m_k}+1}$ are the same as those crossing in the $\inlinecross$-tile corresponding to $\binom{r_{m_{k+1}}}{a_{m_{k+1}}}$ (note these $\inlinecross$-tiles could coincide).
	Repeat (2) using $k+1$.
\end{enumerate}
Call the output $L_{ij}\binom{\textbf{r}}{\textbf{a}}$.
\end{definition}

It is far from obvious that this map is well-defined, that it eventually terminates or that it is invertible.
The fact that the compatible sequence property is preserved follows from the fact that $L_{ij}$ preserves descent sets of reduced words~\cite[Property~1]{little2003combinatorial} and increments values.
We now give proofs of Propositions~\ref{p:Little-Grassmannian} and~\ref{p:hamaker-young}.

\begin{proof}[Proof of Proposition~\ref{p:Little-Grassmannian}]
	This is a straightforward consequence of~\cite[Property~2]{little2003combinatorial}, which says such a sequence of inverse or downward Little bumps exists for every reduced word.
	To see why, let $w \in S_m$.
	Map $\textbf{a} \mapsto \textbf{b}$ via $a_i \mapsto b_i = m-a_i$.
	This map preserves the property that $\textbf{a}$ is a reduced word of a Grassmannian permutation.
	Via this map, a downward Little bump on $\textbf{b}$ corresponds to an upward Little bump on $\textbf{a}$.
	Therefore, the sequence of downward Little bumps mapping $\textbf{b}$ to Grassmannian reduced word $\textbf{b}'$ corresponds to a sequence of (upward) Little bumps mapping $\binom{\textbf{r}}{\textbf{a}}$ to $\binom{\textbf{r}}{\textbf{a}'}$ where $\textbf{a}'$ is a Grassmannian reduced word.
\end{proof}

\begin{proof}[Proof of Proposition~\ref{p:hamaker-young}]
	This is the biword version of {\cite[Prop.~1]{hamaker2014relating}}, as applied to column Edelman-Greene insertion.
	The original statement is for row Edelman-Greene insertion, which differs by transposing the insertion and recording tableaux.
	To upgrade to biwords, note we are inserting the same sequence $\textbf{r}$ into the same standard tableaux.
	The result now follows.
\end{proof}

\subsection{Huang bumps}
\label{ss:huang-A}
An \emph{almost bumpless pipe dream} is a bumpless pipe dream or a diagram with at most one $\inlinebump$, but otherwise satisfies the definition of a bumpless pipe dream.
In order to define Huang bumps, we require two additional maps known as \emph{Monk moves}.

\begin{definition}
Fix an almost bumpless pipe dream. 
\begin{itemize}
\item[(i)] \text{min-droop}: Let $(i,j)$ be a $\inlineare$-turn of a pipe $p$, which could be either a $\inlineare$-tile or a $\inlinebump$-tile. Let $k,\ell \geq 1$ be the smallest numbers such that $(i+k,j)$ and $(i,j+\ell)$ are not $\inlinecross$-tiles. A \emph{min-droop} at $(i,j)$ droops pipe $p$ into $(i+k,j+\ell)$.
\item[(ii)] \text{cross-bump-swap}: Let $(i,j)$ be a $\inlinebump$-tile of pipes $p$ and $q$ and suppose $p$ and $q$ cross each other at $(k,\ell)$. A \emph{cross-bump-swap} move at $(i,j)$ swaps the $\inlinebump$-tile at $(i,j)$ with the $\inlinecross$-tile at $(k,\ell)$.
\end{itemize}
\end{definition}

We are now prepared to define a Huang bump.

\begin{definition}
\label{d:huang}
Let $w$ be a permutation such that $w(i) > w(j)$ for some $i < j$. Given $B$ in $\text{BPD}(w)$, the following gives $H_{ij}(B)$.
\begin{itemize}
\item[(1)] Initialize $(x,y)$ as the $\inlinecross$-tile where pipes $i$ and $j$ cross. Replace this tile with a $\inlinebump$-tile. 
\item[(2)] Perform a min-droop at $(x,y)$ and let $(x',y')$ record the southeast corner of this droop. 
\begin{itemize}
\item[(a)] If $(x',y')$ is a $\inlinejay$-tile then update $(x,y)$ to be the position of the $\inlineare$-tile of pipe $i$ in row $x'$ and repeat Step 2.
\item[(b)] If $(x',y')$ is a $\inlinebump$-tile, let $k$ be the pipe of the $\inlineare$-turn of $(x',y')$. 
\begin{itemize}
\item[(i)] If $i$ and $k$ cross each other at $(x'',y'')$ then perform a cross-bump-swap at $(x',y')$. Update $(x,y) = (x'',y'')$ and repeat Step 2. 
\item[(ii)] If $i$ and $k$ never cross each other then replace $(x',y')$ with a $\inlinecross$-tile and stop.
\end{itemize}
\end{itemize}
\end{itemize}
\end{definition}

To prove Lemma~\ref{l:huang-vexillary}, we require a property of vexillary bumpless pipe dreams.

\begin{lemma}[{\cite[Lem.~7.2]{weigandt2021bumpless}}] \label{p:BPDcrossings}
Suppose $v$ is vexillary and let $\rho(v)$ denote the smallest partition such that $Y_{\rho(v)} \supseteq D_v$. Then one has $B(\inlinecross) \cap Y_{\rho(v)} = \varnothing$ for each $B$ in $\text{BPD}(v)$.
\end{lemma}

\begin{lemma}
\label{l:huang-vexillary}
If $B$ and $B' = H_{ij}(B)$ are bumpless pipe dreams for vexillary permutations, then $B$ and $B'$ agree on $Y_{\rho(v)}$.
\end{lemma}

\begin{proof}
In $H_{ij}$ we initialize at a $\inlinecross$-tile and perform basic Monk moves that only have an affect on tiles outside of $Y_{\rho(v)}$ by Lemma~\ref{p:BPDcrossings}. 
\end{proof}

\bibliographystyle{plain} 
\bibliography{references}

\end{document}